\documentclass[bj, preprint]{imsart}

\usepackage{amsthm,amsmath}
\usepackage[numbers]{natbib}
\usepackage{todonotes,comment,amssymb,mathrsfs,color}
\RequirePackage[colorlinks,citecolor=blue,urlcolor=blue]{hyperref}

\startlocaldefs

\newtheorem{defin}{Definition}[section]
\newtheorem{prop}[defin]{Proposition}
\newtheorem{theorem}[defin]{Theorem}
\newtheorem{coroll}[defin]{Corollary}
\newtheorem{lemma}[defin]{Lemma}
\newtheorem{rem}[defin]{Remark}
\newtheorem{assumption}[defin]{Assumption}

\endlocaldefs

\begin{document}

\begin{frontmatter}

\title{A Feynman-Kac result via Markov BSDEs with generalised drivers}

\runtitle{A Feynman-Kac result via Markov BSDEs with generalised drivers}

\begin{aug}
  \author{\fnms{Elena}  \snm{Issoglio}\thanksref{a}\ead[label=e1]{e.issoglio@leeds.ac.uk}}
  \and
  \author{\fnms{Francesco} \snm{Russo}\thanksref{b,e2}\ead[label=e2,mark]{francesco.russo@ensta-paris.fr}}

\address[a]{Department of Mathematics, University of Leeds, Leeds, LS2 9JT, UK.
\printead{e1}}
\address[b]{ENSTA Paris, 
Institut Polytechnique de Paris, Unit\'e de Math\'ematiques appliqu\'ees,  828, boulevard des Mar\'echaux, F-91120 Palaiseau, France.
\printead{e2}}

\runauthor{E. Issoglio and F. Russo}

\affiliation{University of Leeds and ENSTA Paris}

\end{aug}

\begin{abstract}
 In this paper we investigate BSDEs  where the driver contains a distributional term (in the sense of generalised functions) and derive   general Feynman-Kac formulae related to these BSDEs. We introduce an integral operator to  give  sense to the equation and then we show the existence of a 
strong solution employing results on a related PDE.
Due to the irregularity of the driver, the $Y$-component of a couple $(Y,Z)$ solving the BSDE is not necessarily a semimartingale but a weak Dirichlet process. 
\end{abstract}




\end{frontmatter}

\section{Introduction}\label{sc: introduction}

In this paper we consider Markov backward stochastic differential equations (BSDEs) where the driver is a generalised function  (Schwartz distribution), and investigate Feynman-Kac type formulae in this general setting.  
The classical notion of Brownian BSDE was introduced  in 1990 by E.\ Pardoux and
S.\ Peng in \cite{PardouxPeng90}, after an early work
of J.\ M.\ Bismut in 1973 in \cite{bismut}. It is  a  stochastic differential
equation with prescribed  terminal condition 
$ \xi$ and  driver $\hat f$ expressed by
\begin{equation}\label{BSDEIntroN}
Y_t = \xi + \int_t^T \hat f\left(r,\omega,Y_r,Z_r\right)\mathrm dr  -\int_t^T Z_r \mathrm dW_r.
\end{equation} 
 The unknown is a couple $(Y,Z)$ of adapted processes. 
Existence and uniqueness of the solution for the above equation was established first
supposing (essentially)  only Lipschitz conditions on the driver 
$ \hat f$ with respect to the $y$ and $z$ variables.
 In subsequent works those conditions were
considerably relaxed, see \cite{PardouxRascanu}
 and references therein for recent contributions on the topic.
Ever since the earliest papers, the field of  BSDEs has attracted the interest of a wide number of mathematicians. This is due to the fact that BSDEs turned out to be powerful tools that allowed new and unexpected applications.

Of particular interest
 is the case where the randomness of the
   driver in \eqref{BSDEIntroN}
  is expressed through a forward diffusion process $X$
  and the terminal condition only depends on $X_T$.
We denominate this situation as the {\it Markov case}.
In the present paper we consider the Markov case where the randomness of the driver $\hat f$ depends only on the Brownian motion $W(\omega)$. The key novelty is that $\hat f$ has a linear part in $Z$ of the form $ Z_r b(r,W_r(\omega))$ where $b$ is a suitable generalised function. In particular, we consider BSDEs  of the form
\begin{equation}\label{eq: BSDE classical form}
Y_t = \xi  +\int_t^TZ_r  b(r, W_r) \mathrm d r  + \int_t^T f(r, W_r, Y_r, Z_r)  \mathrm d r -  \int_t^T Z_r \mathrm d W_r.
 \end{equation}
We are interested in  a class of coefficients $b$ of distributional type, namely  
$$b\in C([0,T]; H^{-\beta}_{q} (\mathbb R^d; \mathbb R^d)),$$  for some $\beta\in (0, 1/2)$.  
The objects appearing in \eqref{eq: BSDE classical form} take values in the following sets: $t\in [0,T]$, $\xi, W, Y \in\mathbb R^d$, $Z \in\mathbb R^{d\times d}$ and $f(t, W, Y , Z) \in \mathbb R^d$ (all vectors being column vectors). Here $\xi= \Phi(W_T)$ for some deterministic function $\Phi$. As an example of generalised function  $b$ which is allowed here,  one can think of  the derivative of a H\"older continuous function with H\"older parameter larger than $\frac{1}{2}$ (plus some growth condition at infinity).

Our motivation for looking at these very irregular coefficients comes both from applications and from theoretical issues. Indeed, BSDEs like \eqref{eq: BSDE classical form} and variations of those equations with the same low regularity of coefficients, arise from vastly different contexts from  pricing and hedging problems, to  stochastic control, to probabilistic representation of PDEs.  Below we illustrate some examples of applications of the BSDE \eqref{eq: BSDE classical form} with distributional driver.  
\begin{itemize}
\item BSDEs intervene classically in financial modelling, see e.g.\
\cite{ElKaroui-et.al.97}.
If $\xi$ is a contingent claim based
on some asset price $X$ (already discounted),
then the  {\it price} and the {\it self-financing strategy}
at time $t$ are provided by the couple $(Y_t,Z_t)$
which fulfills
\begin{equation}\label{eq: E2}
Y_t = \xi - \int_t^T Z_r \mathrm dX_r.
\end{equation}
{An interesting case concerns the hedging problem when the underlying $X$ is not a semimartingale, even though  Delbaen \& Schachermayer's fundamental theory imposes that $X$ is a semimartingale if no arbitrage is to be excluded.
However, these no-arbitrage issues can be solved by imposing extra constraints on the class
of admissible strategies. For example, \cite{cheridito03} considered a model driven by fractional Brownian motion (which is not a semimartingale): there arbitrage was prevented by not allowing continuous trading.
In that context, the integral in \eqref{eq: E2} obviously exists
because the strategy processes are of bounded variation.
However in general, a fundamental issue is that  the integral in \eqref{eq: E2} has to be suitably defined. For instance in  \cite{crnsm2}, where $X$ is a finite quadratic variation process (but non necessarily a semimartingale), the integral in \eqref{eq: E2} is a forward integral, and no-arbitrage is guaranteed by appropriately restricting the class of admissible strategies. 
}
Suppose now that the asset price is modelled by the {\it rough process}
$$ X_t = W_t - \int_0^t b(s,W_s)\mathrm ds, $$
where $b(s,\cdot)$ is a Schwartz distribution. Then in this case \eqref{eq: E2} reduces to  BSDE \eqref{eq: BSDE classical form} with $f\equiv 0$, that is 
\begin{equation*}
Y_t = \xi - \int_t^T Z_r\mathrm dW_r + \int_t^T Z_r b(r,W_r) \mathrm dr.
\end{equation*}
Note that the latter integral has still to be defined. 
\item BSDEs are also powerful tools that help to solve stochastic control problems. For example, suppose that  $X$ follows a  stochastic controlled dynamics
\begin{equation}\label{eq:SDEstoch_contr} 
dX_t = \mu(t, X_t, \alpha_t)\mathrm dt + \sigma(t, X_t, \alpha_t) \mathrm dW_t,
\end{equation}
where $\alpha$ is the control process that acts on the drift and the volatility.  Let $d=1$ for simplicity. Suppose that we are interested in maximising the functional
$
J(\alpha)= \mathbb E[\Phi(X_T)]
$
as a function of the control $\alpha$. It is known that this stochastic control problem can be solved with the help of the stochastic maximum principle (Pontryagin maximum principle), see e.g.\ \cite[Section 6.4.2]{pham}. In this setting, one needs to solve
 a BSDE,  called  \emph{adjoint equation}, where the driver is the derivative of the Hamiltonian $ \mathcal H (t,x,a,y,z) := \mu(t,x,a) y + \sigma(t, x, a) z$, that is, the BSDE takes the form
\begin{equation}\label{eq: BSDE stoch control}
-\mathrm dY_t = \mathcal D_x \mathcal H(t, X_t, \alpha_t, Y_t, Z_t) \mathrm dt -Z_t \mathrm dW_t,
\end{equation}
with terminal condition $Y_T = \mathcal D_x \Phi(X_T)$. Here $\mathcal D_x \mathcal H$ denotes the derivative of $\mathcal H$ with respect to the variable $x$.
 It is clear that if $x\mapsto \sigma(t,x,a)$   is a continuous function  which is not differentiable, then the driver of the BSDE will contain some {\em singular} elements.
 More specifically,
 consider for instance the case when $\sigma(t,x,a) = \sigma_0(t,x)\sigma_1(a) $ where $\sigma_1(\cdot)$ is  bounded 
 and
$  \sigma_0(t, \cdot)\in H^s_p(\mathbb R)$ where $s<1$. Then $ \mathcal D_x \sigma_0(t, \cdot)
 \in H^{s-1}_p(\mathbb R)$ and $s-1<0$, that is a generalised function like $b$ in \eqref{eq: BSDE classical form}.
Indeed, in this case we recover a BSDE where there is  a rough part linear in $z$, namely $ \mathcal D_x  \sigma(t,x,\alpha(t,x)) z  = : b(t,x)  z$, much like \eqref{eq: BSDE classical form} with $-\beta=s-1$. 
{We remark that any $s'$-H\"older continuous function $\sigma_0(t, \cdot)$ with compact support belongs to the fractional Sobolev space $H^s_p(\mathbb R)$ for any $s<s'$ and $p\geq2$, see e.g. \cite[Proposition 4.1]{issoglio13}. 
Moreover, if the diffusion coefficient  is $s'$-H\"older continuous with $s'\geq\tfrac12$ and if the drift $\mu$ is Lipschitz, one can show pathwise uniqueness of a solution $X$ to \eqref{eq:SDEstoch_contr} for every given control $\alpha$ following the proof of \cite[Proposition 2.13, Ch 5]{karatzasShreve}.
 In finance, }
such kind of non-smooth volatility $\sigma$
 can be obtained if one looks for example at CIR models with {\em uncertain volatility}, where  $\mu(t,x,a) = bx+c$ and   $\sigma(t,x,a)= \sqrt x a$. Here the control $a$ is a scaling parameter that represents the uncertainty of the volatility and varies between two given values $a_1, a_2$. 
\item As we mentioned earlier, another main application of BSDEs is their use in providing  probabilistic representations to the solution of certain non-linear PDEs. It is known  (at least in the classical case) that when $\xi= \Phi(W_T)$, then BSDE \eqref{eq: BSDE classical form} is linked to a PDE of the form 
\begin{equation}\label{eq: PDE intro}
\left\{\begin{array}{l}
\partial_t u +\frac12 \Delta u ={-\nabla u^* \, b- f(\cdot, u, \nabla u)}\\
  u(T) = \Phi,
\end{array}
\right.
\end{equation}
see Section \ref{sc: preliminaries} for details about the notation. 
If $u$ is the solution of the PDE, then $ Y_t:= u(t,W_t) $ and $ Z_t:= \nabla u^*(t, W_t)$ is a solution to the  BSDE \eqref{eq: BSDE classical form}.  We emphasize that $(Y,Z)$ is a  strong solution to the BSDE related to the Brownian filtration related to $W$, which is then used to represent the solution $u$ to the PDE via non-linear Feynman-Kac type formulae.  Note that  if we were to work with SDEs with distributional coefficients, we would have representation of the (linear) PDE via weak solutions and not strong solutions, because in this case the solution to the SDE is weak, see \cite{flandoli_et.al14}. 
\end{itemize}

Motivated by these examples, we study   BSDE \eqref{eq: BSDE classical form} and the Feynman-Kac representation of the solution to   PDE \eqref{eq: PDE intro} from a theoretical perspective. 
PDEs with distributional coefficients appear naturally as Fokker-Planck
type equations for diffusions in irregular medium or polymers,
see e.g. \cite{mathieu, seignourel, diel}.
The topic of stochastic differential equations involving distributional coefficients has attracted a lot of interest, in particular for (forward) SDEs. See for example  \cite{ew, frw1, frw2} in the case where the solution is not a semimartingale. See also \cite{russo_trutnau07} and more recently \cite{flandoli_et.al14, diel}. 
 For what concerns the case of backward SDEs involving a distribution  we mention the  works \cite{erraoui-ouknine-sbi98}  on (reflected) BSDEs with distribution as terminal condition, and \cite{russo_wurzer17} whose authors
 studied a one-dimensional BSDE (with random terminal time)
 involving distributional coefficients via a forward stochastic process. 
In \cite{diehl-zhang17} they considered BSDEs where the driver is a Young integral.
Recently \cite{issoglio_jing16, barrasso2} studied Markov BSDEs with special forward process with distributional drift, using different techniques than ours.

In this paper we make a substantial step towards a deeper understanding of backward equations with distributional drivers and their link to rough non-linear PDEs expressed via Feynman-Kac type formulae. 
It is worth noticing that even though one expects that BSDE \eqref{BSDEIntroN} is somehow equivalent to   PDE \eqref{eq: PDE intro}, this is a priori not clear 
 in the singular case when $b$ is a distribution. We  rigorously prove this fact in the present paper.
Our idea is to give an intrinsic meaning to the distributional term $Z_rb(r,W_r)$ in order to define and  solve the BSDE. We start by introducing an integral operator $A^{Y,W}$ (see Definitions \ref{def: operator A YW in C} and \ref{def: operator A YW in H general}) that will provide a proper mathematical meaning to the term $\int_t^T Z_rb(r,W_r) \mathrm dr $ when evaluated in $b$. This operator is defined in terms of the Brownian motion $W$ and a process $Y$. 
{In the special case when $Y=W$, $A^{W,W}$ will be denominated
  as the {\it  occupation time operator}, since $A^{W,W}(g)$ can be linked to the occupation time formula,  see Remark \ref{rem:ustunel}.  
}
 Using the integral operator $A^{Y,W}$ we introduce an equivalent formulation of the BSDE  \eqref{BSDEIntroN} (see Definition \ref{def: solution BSDE integral A}) and show that it extends  the classical notion of solution from Pardoux-Peng, see Proposition \ref{pr: equivalence of integral BSDE formulation}. In Proposition \ref{pr: continuity of A WW in C} we show that the occupation time operator $A^{W,W}$  is well-defined for $b$s in a specific class of distributions, namely in the fractional Sobolev space $H^{-\beta}_q$ where the parameters satisfy Assumption \ref{ass: parameters}.
In  Proposition \ref{pr: chain rule in H} we  show a chain rule for $\phi(t, W_t)$ for a certain class of  $\phi\in C^{0,1}$ (related to the heat equation \eqref{eq: PDE for phi}), and  the remainder in the chain rule is expressed in terms of the occupation time operator $A^{W,W}$. 
Our main results  are Theorem \ref{thm: Markovian BSDE - existence of sol}, where we prove  the existence of a solution to the BSDE \eqref{eq: BSDE integral A} in the Markovian framework given in terms of the solution of PDE \eqref{eq: PDE intro}, and Corollary \ref{cor: Feynman-Kac repr}, which is the Feynman-Kac formula for the probabilistic representation of the solution of the PDE. We also investigate uniqueness of the solution of the BSDE in a particular class (Proposition \ref{prop: Markovian BSDE - uniq of sol}).

The paper is organised as follows. In Section \ref{sc: preliminaries} we recall useful results, set the  notation and state the assumptions needed later on. In 
Section \ref{sc: alternative BSDE} we define the integral operator $A^{Y,W}$ and  introduce the equivalent formulation of the BSDE.  Section \ref{sc: PDE results} collects important analytical properties of the PDE associated to the BSDE in  the Markovian case. 
In Section \ref{ssc: preliminaries AYW} we investigate the properties of the occupation time operator and in Section \ref{ssc: existence} we state and prove the main results of existence of a solution to the BSDE   and  the corresponding Feynman-Kac formula. 
Finally in  Appendix \ref{sc: appendix} we state and prove a  technical result needed in the paper, as well as two technical proofs which have been moved here for ease of reading.

\section{Preliminaries and notation}\label{sc: preliminaries}
  
Throughout the paper $c$ and $C$ denote positive constants whose specific value is not important and may change from line to line.

\subsubsection*{Function spaces - notation} We denote by $C^{0,1}([0,T]\times \mathbb R^d)$ the space of real-valued continuous functions on $[0,T]\times \mathbb R^d$ which are continuously differentiable in the variable $x\in \mathbb R^d$. By $\varphi_n\to 0$ in $C^{0,1}$ we mean that $\varphi_n$ and $\nabla \varphi_n$ (the gradient taken w.r.t.\ the $x$-variable) converge to 0 uniformly on compacts. The space $C^{0,1}$ is then endowed with the topology related to this convergence.  
For a vector $\varphi= (\varphi_1, \ldots, \varphi_d)$ such that $\varphi_i\in C^{0,1}([0,T]\times \mathbb R^d)$ for all $i$, we write  $\varphi\in C^{0,1}([0,T]\times \mathbb R^d; \mathbb R^d)$ or $\varphi\in C^{0,1}$ for shortness.
Similarly we denote by $C^{1,2}([0,T]\times \mathbb R^d)$ the space of real-valued  functions on $[0,T]\times \mathbb R^d$ which are continuously differentiable once in $t$ and twice in $x$, and by  $C^{1,2}:= C^{1,2}([0,T]\times \mathbb R^d; \mathbb R^d)$. 
 The topology is similar to the one for $C^{0,1}$.
 Moreover we use  $C_c (\mathbb R^d)$ to denote the space of  continuous functions of $x$ with compact support and  $C_c^\infty(\mathbb R^d) $ to denote the space of  infinitely differentiable functions with compact support. Again the short-hand notation for $ \mathbb R^d$-valued functions is $C_c := C_c (\mathbb R^d;\mathbb R^d)$ and $C_c := C_c^\infty (\mathbb R^d;\mathbb R^d)$.
The Euclidean norm in $\mathbb R$ and $ \mathbb R^d$, and the Frobenius norm in $\mathbb R^{d\times d}$ will be denoted by $|\cdot|$. For a vector $v$, its transpose is denoted by $v^*$. If $v$ is a real-valued function of $x\in\mathbb R$ then $\nabla v^*$ denotes the transpose of the column vector $\nabla v$. Moreover is $u$ is a vector-valued function of $x$ then $\nabla u$ is a matrix where the $j$-th column is given by $\nabla u_j$ so that $(\nabla u )_{i,j} =\frac{\partial} {\partial x_i} u_j$. For the matrix $\nabla u$, we denote its transposed by $\nabla u^*$.
  
\subsubsection*{Stochastic analysis tools}   Throughout the paper    
 $(\Omega,  \mathcal G, P)$ is a probability space on which  a $d$-dimensional Brownian motion $W:=(W_t)_t$ is defined,  with Brownian filtration $\mathcal F :=(\mathcal F_t)_t $. 
   
  We denote by $\mathcal C$  the space of continuous stochastic processes indexed by $[0,T]$  with values in $\mathbb R^d$. In this space we will consider {u.c.p.}\ convergence (uniform convergence in probability) for stochastic processes. More precisely, we say that a family of stochastic processes $X^n$ indexed by $[0,T]$ converges u.c.p.\  to $X$ in $\mathcal C$ if 
  \[
\sup_{s\in[0,T]} \vert X^n_s -X_s\vert \to 0 \textrm{ in probability.}
  \]

The following definitions  of covariation process and  weak-Dirichlet process  are taken from \cite{gozzi_russo06}, see also \cite{russo_vallois95} for more details.

Given two stochastic processes $Y:=(Y_t)_t$ and $X:=(X_t)_t$, we denote by $[Y,X]$ the \emph{covariation process} of $Y$ and $X$ which is  defined by
\[ 
[Y,X]_t := \lim_{\varepsilon \to 0} \frac1\varepsilon \int_0^t (Y_{s+\varepsilon} -Y_s) (X_{s+\varepsilon} -X_s) \mathrm ds,
\]	
if the limit exists in the u.c.p.\  sense in $t$. 
If $X,Y$ are $d$-dimensional processes then $[Y,X]\in \mathbb R^{d\times d}$ is the tensor covariation and it is defined component by component by  $([Y,X])_{i,j} = [Y_i, X_j]$, if it exists. {Note that the covariation is not symmetric because the matrix does not need to be squared and in particular we have $[Y,X] = [X,Y]^*$.} This concept extends the  classical covariation of continuous semimartingales. We remark that the covariation of a bounded variation process and a continuous process is always zero.  
{\begin{defin}\label{def:weak_dirichlet_process}
Given a filtration $\mathcal F:=(\mathcal F_t)_t$, a real process $D$ is said to be an \emph{$\mathcal F$-weak Dirichlet} process if it can be written as $D=M+A$ where 
\begin{itemize}
\item[(i)] $M:=(M_t)_t$ is an $\mathcal F$-local martingale,
\item[(ii)] $(A_t)_t$ is a  \emph{martingale-orthogonal process}, namely a process such that $[A,N]=0$ for every $\mathcal F$-continuous local martingale $N$. For convenience we also set $A_0=0$. 
\end{itemize}
\end{defin}
Note that in \cite{gozzi_russo06} they use the name \emph{weak zero energy process} for the  \emph{martingale-orthogonal process}.}
 It was shown that the decomposition $D=M+A$ is unique and every $\mathcal F$-semimartingale is an $\mathcal F$-weak Dirichlet process. 
A vector $D=(D^1, \ldots, D^d)$ is an  $\mathcal F$-weak Dirichlet process if every component $D^i$ is an $\mathcal F$-weak Dirichlet process.  
We will drop the $\mathcal F$ and simply write weak Dirichlet process when it is clear what filtration $\mathcal F$ we are considering.

\begin{prop}\label{pr: covariation}
  Let $v\in C^{0,1}([0,T]\times \mathbb R^d)$ and $S^1$ (resp.\ $S^2$)  be an $\mathbb R^d$-valued  (resp.\ $\mathbb R$-valued) continuous   $\mathcal F$-semimartingale with martingale component $M^1$ (resp.\ $M^2$). Then 
  \begin{equation}\label{eq: covariation}
[v(\cdot, S^1), S^2]_t = \int_0^t   \nabla v^* (r, S^1_r)  \mathrm d[M^1, M^2 ]_r.  
  \end{equation}
  \end{prop}
\begin{proof}
Let us denote by $M^v_t:=\int_0^t  \nabla v^* (r, S^1_r)  \mathrm d M^1_r $. By \cite[Corollary 3.11]{gozzi_russo06} we have that $v(\cdot, S^1)$ is a weak Dirichlet  process with martingale component $M^v$.
If $A^v$ is the related martingale-orthogonal   process, we know that $[A^v, N] = 0$ for any $\mathcal F$-continuous local martingale $N$, 
 see \cite[Proposition 1.7.(b)]{RVSem}.
 Consequently the left-hand side of \eqref{eq: covariation} gives 
\begin{align*}
[v(\cdot, S^1), S^2]_t&=[M^v, M^2]_t\\
&=\left[ \int_0^\cdot  \nabla v^*(r, S^1_r)  \mathrm d M^1_r, M^2\right]\\
&=\int_0^t  \nabla v^* (r, S^1_r)  \mathrm d[M^1, M^2 ]_r,  
\end{align*}
where the last equality holds true because the covariation $[\cdot, \cdot]$ extends the one of semimartingales.
\end{proof} 
When $v$ is a vector-valued function (say $u$), the covariation becomes a matrix and an analogous result holds, as stated in the corollary below (in the special case when $u$ is a function of Brownian motion).
\begin{coroll}\label{cor: covariation}
  Let $\phi\in C^{0,1}([0,T]\times \mathbb R^d; \mathbb R^d)$, $W$   be an $\mathbb R^d$-valued $\mathcal F$-Brownian motion  and $N$ an   $\mathcal F$-continuous local martingale with values in $\mathbb R^d$. Then 
  \begin{equation*}\label{eq: covariation Rd}
[\phi(\cdot, W),  N ]_t = \int_0^t \nabla \phi^*  (r, W_r) \mathrm d  [W, N]_r. 
\end{equation*} 
\end{coroll}

\subsubsection*{Heat semigroup and fractional Sobolev spaces}
We denote by $\mathcal S(\mathbb R^d)$ the space of $\mathbb R^d$-valued Schwartz functions and by  $\mathcal S'(\mathbb R^d)$ the space of Schwartz distributions.  
Setting $A:= I-\frac12 \Delta$, we can view this as an operator in $\mathcal S' (\mathbb R^d)$. Then its fractional powers $A^\alpha$ are well-defined on the same space for any power $\alpha\in \mathbb R$ by means of Fourier transform (see e.g.\ \cite[Remark 1.2 (iii)]{triebel10}). One can define the classical fractional Sobolev spaces via these fractional powers, that is $H^s_r(\mathbb R^d):= A^{-s/2}(L^r(\mathbb R^d))$. These are Banach spaces endowed with the norm $\|u\|_{H^s_r}:=\|A^{s/2}u\|_{L^r}$. It is also known  that $A^{-\alpha/2}$ is an isomorphism between $H^s_r(\mathbb R^d)$  and $H_r^{s+\alpha}(\mathbb R^d)$, for each $\alpha\in \mathbb R$,  (see again \cite[Remark 1.2 (iii)]{triebel10}).  
 We denote by  $(P(t), t\geq0)$ the heat semigroup
 which acts on any $L^r(\mathbb R^d)$  for $1<r<\infty$,
 with kernel $p_t(x)= \frac{1}{(2\pi t)^{d/2}} \exp \left( -\frac{|x|^2}{2t}\right) $. This    is a bounded analytic semigroup generated by $\tfrac12 \Delta$, see \cite[Theorems 1.4.1, 1.4.2]{davies89}.
 We denote by  $(S(t), t\geq 0)$ the semigroup given by $S(t):= e^{-t}P(t)$. If we consider $A$ as an unbounded operator on $L^r(\mathbb R^d)$, then it is 
well-known that the semigroup $S$ is generated by $-A$ and $D(A) = H^2_r(\mathbb R^d)$. 
 Fractional powers of $A$, as unbounded operator on $L^r(\mathbb R^d)$,
 where $-A$ is the generator of an analytic
 semigroup can also be defined (see \cite[Section 2.6]{pazy83})  and a key fact  that links these operators with fractional Sobolev spaces is that  $D(A^{s/2}) = H^s_r(\mathbb R^d) $, which follows from interpolation theory.\footnote{This can be seen by applying \cite[Theorem 1.15.3]{triebel78} with $\alpha=0,\, \beta=1$ and $0<\theta<1$ to get $D(A^\theta) = [D(A^0), D(A^1)]_\theta$. The latter is equal to  $[L^p, H^2_p]_\theta$ because of known result on the operator $A$ and its integer powers. Finally using \cite[Theorem 2.4.2/1]{triebel78} with $q_0=q_1=q=2,\, p_0=p_1=p$ and $s_0=1, s_1=2$ (so that $s=2\theta$) one gets $[L^p, H^2_p]_\theta = [H^0_p, H^2_p]_\theta = H^{2\theta}_p$.} 
Using this and the isomorphism property one has for $\delta>\beta>0, \delta+\beta<1$ and $0<t\leq T$ that  $P(t): H^{-\beta}_r(\mathbb R^d) \to H^{1+\delta}_r(\mathbb R^d) $ for all $1<r<\infty$
and
\begin{equation}\label{eq: mapping prop Pt}
\|P(t)w\|_{H_r^{1+\delta}(\mathbb R^d)}\le C e^t  t^{-\frac{1+\delta+\beta}{2}}\|w\|_{H_r^{-\beta} (\mathbb R^d)},
\end{equation}
for $w\in H^{-\beta}_r(\mathbb R^d), t>0$. This follows from a similar property for the bounded analytic semigroup   $S$ which is stated in \cite[Lemma 10]{flandoli_et.al14}, see also  \cite[Proposition 3.2]{issoglio13} for the analogous on domains $D\subset \mathbb R^d$.
Moreover it is easy to show\footnote{This   can be seen by writing $P(t) = e^t e^{-t} P(t) = e^{t}S(t)$. Since  $-A$ is the generator of  $S$ we have that $S(t): L^r \to D (A^{s/2}) $ by  \cite[Chapter 2, Thm 6.13 (a)]{pazy83}. Moreover $D(A^{s/2})= H^s_r$ as recalled above. Let $w\in H^s_r$, so we also have $w\in L^r$ thus $S(t)w\in H^s_r$. Then by the definition of norm in $H^r_s$ we get $\| P(t)w\|_{H^s_r}=e^t \|S(t)w\|_{H^s_r} = e^t \|A^{s/2} S(t)w\|_{L^r} $. Now applying \cite[Chapter 2, Thm 6.13 (b)]{pazy83} we know that $A^{s/2} $ and $S(t)$ commute and using the contractivity of $P(t)$  on $L^r$ we get  $ \|P(t) w \|_{H^s_r} \leq \|e^t  S(t) A^{s/2} w\|_{L^r} \leq  \|A^{s/2} w\|_{L^r} $ and the latter is equal to $\|w\|_{H^s_r} $ by definition of the norm. } that  $P(t)$ is a contraction on $H^s_r(\mathbb R^d)$ for all $1<r<\infty$ and $s\geq0$, that is for all $w \in H_r^{s}(\mathbb R^d)$ we have
\begin{equation}\label{eq: P contraction} 
\|P(t)w\|_{H_r^{s}(\mathbb R^d)}\le  \|w\|_{H_r^{s}(\mathbb R^d)}.
\end{equation}

As done already before in this paper, we denote by  $H^s_r$ the spaces  $H^s_r(\mathbb R^d; \mathbb R^d)$, whose definition is as above for each component. {Note that by slight abuse of notation the same  $H^s_r$ might be the space  $H^s_r(\mathbb R^d; \mathbb R^{d\times d})$, especially when considering functions like $\nabla u$.}   
When we write $u\in H^s_r $ we mean that each component $u_i$ is in $H^s_r(\mathbb R^d)$. The norm will be denoted with the same notation for simplicity. 
 
 \subsubsection*{Pointwise product}   Here we recall the definition of the {\em pointwise product} between a function and a distribution, for more details see \cite{runst_sickel96}. Let $g\in \mathcal{S}^\prime(\mathbb R^d)$. We choose a function $\psi\in \mathcal S (\mathbb R^d)$ such that  $0\le \psi(x)\le 1 $, for every $x\in \mathbb R^d$  and 
 \begin{equation*}
 \psi(x)=\left\{
 \begin{array}{ll}	
1, &\quad|x|<1,\\
0, &\quad|x|\ge 2.
 \end{array}\right.
 \end{equation*}
 For every $j\in \mathbb{N}$, we consider the approximation $S^j g$ of $g$ as follows:
 \begin{equation*}
 S^jg (x):=\mathscr{F}^{-1} \left(\psi\left(\frac{\xi}{2^j}\right)\mathscr{F}(g)\right)(x),
 \end{equation*}
 where $\mathscr{F}(g)$ and $\mathscr{F}^{-1}(g)$ are the Fourier transform  and the inverse Fourier transform  of $g$, respectively. 
 The  product $gh$ of $g, h\in \mathcal{S}^\prime(\mathbb R^d)$ is defined as 
 \begin{equation} \label{eq: pointwise product}
 gh:=\lim_{j\to\infty}S^j g S^j h,
 \end{equation} 
 if the limit exists in $\mathcal{S}^\prime (\mathbb R^d)$. 

 \begin{lemma}\label{lm: pointwise product}\cite[Theorem 4.4.3/1]{runst_sickel96}
 Let  $g\in H_q^{-\beta}(\mathbb{R}^d)$, $h\in  H_p^{\delta}(\mathbb R^d)$ for $1<p,q<\infty$, $q>\max(p,\frac{d}{\delta})$, $0<\beta<\frac{1}{2}$ and $\beta<\delta$. Then the pointwise product $gh$ is well-defined, it belongs to the space $ H_p^{-\beta}(\mathbb R^d)$ and we have the  bound
\begin{equation} \label{eq: pointwise product bound}
 \|gh   \|_{  H_p^{-\beta}(\mathbb R^d)} \leq c\| g\|_{ H_q^{-\beta}(\mathbb{R}^d)}   \cdot  \|  h \|_{ H_p^{\delta}(\mathbb R^d)}.
\end{equation}
\end{lemma}
In this paper we will always use this product in such fractional Sobolev 
spaces.

 \subsubsection*{More on function spaces}
We observe that when we talk about smooth drivers we consider elements of $C_c([0,T]\times \mathbb R^d; \mathbb R^d)$ or of  $C_c^\infty ([0,T]\times \mathbb R^d; \mathbb R^d)$, which is  defined to be the space of all $f\in C_c([0,T]\times \mathbb R^d; \mathbb R^d)$ such that $ \frac{\partial^\alpha f}{\partial x^\alpha}$ exists for all  multi-indexes $ \alpha$  and $ \frac{\partial^\alpha f}{\partial x^\alpha} \in C_c([0,T]\times \mathbb R^d; \mathbb R^d)$. It is clear that each function in $ C_c^\infty([0,T]\times \mathbb R^d)$ is an element of $L^r(\mathbb R^d)$ for any fixed time $t\in[0,T]$ and for $2\leq r\leq \infty$, and moreover it is continuous with respect to the topology in $L^r(\mathbb R^d)$. Since $L^r(\mathbb R^d)\subset H^{s}_{r}(\mathbb R^d)$ for $s\leq 0$ we have the inclusion $C_c^\infty([0,T]\times \mathbb R^d; \mathbb R^d)\subset C([0,T]; H^s_r)$.  
 
For the following, see \cite[Section 2.7.1]{triebel78}.  The closures of  $\mathcal S(\mathbb R^d)$
 with respect to the norms 
$$
\|h\|_{C_b^{0,0}(\mathbb R^d)} := \|h\|_{L^\infty(\mathbb R^d)}
$$ 
and 
$$
\|h\|_{C_b^{1,0}(\mathbb R^d)} := \|h\|_{L^\infty(\mathbb R^d)} + \|\nabla h\|_{L^\infty(\mathbb R^d)} 
$$ 
respectively, are denoted by $C_b^{0,0}(\mathbb R^d)$ and $C_b^{1,0}(\mathbb R^d)$. 
For any $\alpha >0$, we  consider the Banach spaces
\begin{align*}
	&C^{0+ \alpha} (\mathbb R^d)= \{ h \in C_b^{0,0}(\mathbb R^d) :  \|h\|_{C^{0+\alpha}(\mathbb R^d)} <\infty \}\\
	&C^{1+\alpha}(\mathbb R^d) =\{  h \in C_b^{1,0}(\mathbb R^d) :  \|h\|_{C^{1+\alpha}(\mathbb R^d)} <\infty \},
\end{align*}
endowed with the norms
\begin{align*}
	&\|h\|_{C^{0+\alpha} (\mathbb R^d)}:= \| h \|_{L^\infty(\mathbb R^d)}  + \sup_{x\neq y \in \mathbb R^d}  \frac{|h(x)-h(y)|}{|x-y|^\alpha} \\
	& \|h\|_{C^{1+\alpha}(\mathbb R^d)}  := \| h \|_{L^\infty(\mathbb R^d)}  + \| \nabla h \|_{L^\infty(\mathbb R^d)} + \sup_{x\neq y \in \mathbb R^d}  \frac{|\nabla h(x)-\nabla h(y)|}{|x-y|^\alpha},
\end{align*}
respectively. We denote by $C^{0+\alpha}$ and $C^{1+\alpha}$  the analogous spaces for $\mathbb R^d$-valued functions and the corresponding norms by $\|\cdot\|_{C^{0+\alpha}}$ and $\|\cdot\|_{C^{1+\alpha}}$.

Let $B$ be a Banach space. We denote by $C([0,T];B)$ the Banach space of $B$-valued continuous functions and its sup norm by $\|\cdot\|_{C([0,T];B)}$. For $h\in C([0,T];B)$ and $\rho\ge 1$ we also use the family of equivalent norms $\{\|\cdot\|_{C([0,T];B)}^{(\rho)}, \rho\ge 1\}$, defined by
\begin{equation}\label{eq: rho equiv norm}
\|h\|_{C([0,T];B)}^{(\rho)}:=\sup_{0\le t\le T} e^{-\rho t} \|h(t)\|_{B}. 
\end{equation}
The following lemma is a fractional Sobolev embedding theorem which will be used several times in this paper. It is a generalisation of the Morrey inequality to fractional Sobolev spaces. For the proof we refer to \cite[Theorem 2.8.1, Remark 2]{triebel78}.

\begin{lemma}[Fractional Morrey inequality]\label{lm: fractional Morrey ineq}
	Let $0<\delta< 1 $ and $d/\delta<r<\infty$. If $h\in H^{1+\delta}_r(\mathbb R^d)$ then there exists a unique version of $h$ (which we denote again by  $h$) such that $h$ is differentiable. Moreover $ h\in C^{1+\alpha}(\mathbb R^d)$ with $\alpha= \delta-d/r$ and
	\begin{equation}\label{eq: fractional Morrey ineq}
	\|h\|_{C^{1+\alpha}(\mathbb R^d)}\leq c \|h\|_{H^{1+\delta}_r(\mathbb R^d)}, \quad  \|\nabla h\|_{C^{0+\alpha}(\mathbb R^d)}\leq c \|\nabla h\|_{H^{\delta}_r(\mathbb R^d)},
	\end{equation}
	where $c=c(\delta, r, d)$ is a universal constant.
In particular $h$ and $\nabla h$ are bounded.	
 
\end{lemma}

\subsubsection*{Assumptions} Later in the paper we will use the following  assumptions about the parameters  and the functions involved. 

\begin{assumption}\label{ass: parameters} We always choose $(\delta, p)\in K(\beta,q)$, where the latter set is defined below in two different cases.
\begin{description}
\item [Case $d\geq2$.] Let $\beta\in\left(  0,\frac{1}{2}\right)  $ and  $q\in\left(  \frac{d}{1-\beta},\frac{d}{\beta}\right)  $.		For given $\beta$ and $q$ as above we define the  set
		\begin{equation}\label{eq: set admissible kappa}
		K(\beta,q):=\left\{(\delta, p)\in \mathbb R^2: \; \beta< \delta< 1-\beta, \,  \frac d\delta < p < q   \right\},
		\end{equation}
		which is drawn in  Figure \ref{fig: kappa}. 

\item [Case $d=1$.] In this case we let $\beta\in\left(  0,\frac{1}{2}\right)  $ and  $q\in\left(  2 ,\frac{1}{\beta}\right)  $. For given $\beta$ and $q$ as above we define the  set
\begin{equation}\label{eq: set admissible kappa d=1}
		K(\beta,q):=\left\{(\delta, p)\in \mathbb R^2: \; \beta< \delta< 1-\beta, \,   \frac1\delta < p < q, 2 \leq p  \right\},
		\end{equation} which is drawn in  Figure \ref{fig: kappa d=1}. 
		
\end{description}		 
\end{assumption}

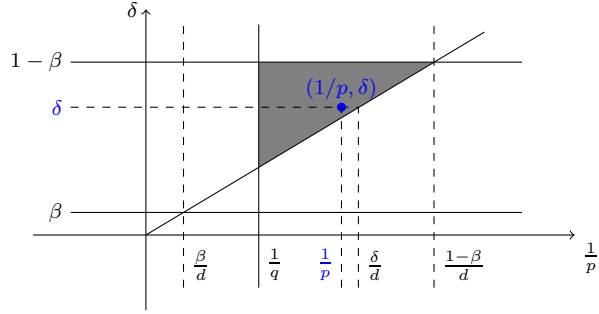
\begin{figure}
\begin{tikzpicture}
	\draw[->] (-1,0)--(6.2,0) node[below right] {$\frac1p$};
	\draw [->] (0.5,-1)--(0.5,3) node[left] {$\delta$};
	\fill[gray] (2, .9)--(2,2.3)--(4.33,2.3)--(2, .9);
	\draw (-.5,.3)  node[anchor=east]{$\beta$}--(5.5,.3);
	\draw[dashed] (-.5,1.7)  node[anchor=east]{\color{blue}$\delta$}--(3.325,1.7);
	\filldraw [blue] (3.1,1.7) circle (1.5pt) node[anchor=south] {\footnotesize \color{blue} $(1/p,\delta) $};
	\draw (-.5,2.3) node[anchor=east]{$1-\beta$} --(5.5,2.3);
	\draw (2,-.7) node[anchor=south west]{$ \frac1q$} --(2,2.8);
	\draw[dashed] (3.1,-.7) node[anchor=south east]{\color{blue}$ \frac1 p$} --(3.1,1.7);
	\draw[dashed] (3.325,-.7) node[anchor=south west]{$ \frac\delta d$} --(3.325,1.7);
	\draw[dashed] (1,-.7) node[anchor=south west]{$ \frac{ \beta}{ d}$} --(1,2.8);
	\draw[dashed] (4.33,-.7) node[anchor=south west]{$\frac{1-\beta}{ d}$} --(4.33,2.8);
	\draw (.5,0)--(5,2.7);  
	\end{tikzpicture}
		\caption{The set $K(\beta, q)$ for $d>1$. Given any couple $\beta, q$ that satisfies the assumptions, the grey region shows all possible $\delta, p$.} \label{fig: kappa}
\end{figure}

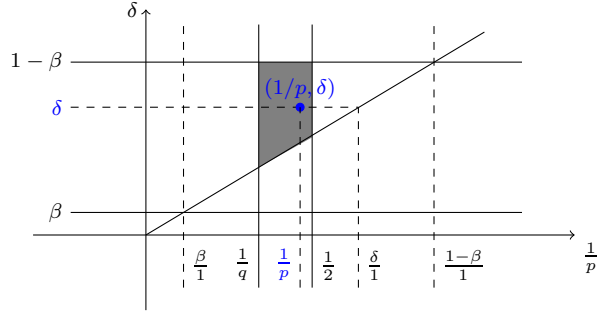
\begin{figure}
	\begin{tikzpicture}
	\draw[->] (-1,0)--(6.2,0) node[below right] {$\frac1p$};
	\draw [->] (0.5,-1)--(0.5,3) node[left] {$\delta$};
	\fill[gray] (2, .9)--(2,2.3)--(2.71,2.3)--(2.71, 1.3);
	\draw (-.5,.3)  node[anchor=east]{$\beta$}--(5.5,.3);
	\draw[dashed] (-.5,1.7)  node[anchor=east]{\color{blue}$\delta$}--(3.325,1.7);
	\filldraw [blue] (2.55,1.7) circle (1.5pt) node[anchor=south] {\footnotesize \color{blue} $(1/p,\delta) $};
	\draw (-.5,2.3) node[anchor=east]{$1-\beta$} --(5.5,2.3);
	\draw (2,-.7) node[anchor=south east]{$ \frac1q$} --(2,2.8);
	\draw[dashed] (2.55,-.7) node[anchor=south east]{\color{blue}$ \frac1 p$} --(2.55,1.7);
	\draw[dashed] (3.325,-.7) node[anchor=south west]{$\frac \delta 1$} --(3.325,1.7);
	\draw[dashed] (1,-.7) node[anchor=south west]{$\frac{ \beta}1$} --(1,2.8);
	\draw[dashed] (4.33,-.7) node[anchor=south west]{$\frac{1-\beta}1 $} --(4.33,2.8);
	\draw (2.71,-.7) node[anchor=south west]{$ \frac12 $} --(2.71,2.8);
	\draw (.5,0)--(5,2.7);  
	\end{tikzpicture}
	 
	\caption{The set $K(\beta, q)$ for $d=1$. Given any couple $\beta, q$ that satisfies the assumptions, the grey region shows all possible $\delta, p$.} \label{fig: kappa d=1}
\end{figure}
Note that $K(\beta,q)$ is non-empty since $\beta<\frac12 $ and $ \frac d{1-\beta}< q < \frac d\beta $. The set $K(\beta, q)$ was first introduced in  \cite{flandoli_et.al14} without the restriction $q,p\geq2$. This is satisfied anyway if $d>1$. If $d=1$ then the set of admissible couples $(\delta, p)$ is shown in Figure \ref{fig: kappa d=1}.

The following  assumption concerns the driver $f$ and the terminal condition $\Phi$ (recall that the terminal condition $\xi$ in the BSDE will be replaced by $\Phi(W)$ in later sections).
\begin{assumption}\label{ass: coefficients}
We assume the following.
\begin{itemize}
\item {$\Phi:\mathbb R^d\to \mathbb R^d$ is an element of $H^{1+\delta+2\gamma}(\mathbb R^d)$ for some $0<\gamma<\frac{1-\delta-\beta}{2}$;} 
\item   $f:[0,T]\times\mathbb R^d \times \mathbb R^d\times \mathbb R^{d\times d} \to \mathbb R^d$ is continuous in $(x, y,z)$ uniformly in $t$, and is Lipschitz continuous in $(y,z)$ uniformly in $t$ and $x$, i.e.\ $\vert f(t,x,y,z)-f(t,x, y', z')\vert\leq L(\vert y-y'\vert+ \vert z-z'\vert )$ for any   $y,y' \in \mathbb R^d$ and $z, z'\in \mathbb R^{d\times d}$. Moreover $f(t,x,0,0)$ is continuous in $(t,x)$; 
\item  $\sup_{t,x} |f(t,x,0,0)| < \infty$ a.s.\ and $ \sup_{t\in[0,T]} \|f(t,\cdot, 0,0)\|_{L^p} < \infty$.
\end{itemize}
\end{assumption}  
  
\section{Alternative representation for the BSDE}\label{sc: alternative BSDE}  
In this section we propose an alternative representation for the BSDE \eqref{eq: BSDE classical form} which turns out to be  well-suited for BSDEs with rough drivers and it is  equivalent to the one above if the driver is smooth, see Proposition \ref{pr: equivalence of integral BSDE formulation} below.

Let  $W=(W_t)_t$ be a $d$-dimensional Brownian motion equipped with its canonical filtration $\mathcal F= (\mathcal F_t)_t $.  
To be able to consider rough drivers, the main term in \eqref{eq: BSDE classical form} that needs  to be (re)defined is the  integral $ \int_t^T Z_r b(r, W_r) \mathrm d r$. Here we recall that $b$ is a column $ \mathbb R^d$-vector  and $Z\in \mathbb R^{d\times d}$ so that the integral is a column vector.
We introduce the following integral operator.
\begin{defin}\label{def: operator A YW in C}
Let   $Y=(Y_t)_t$ be a continuous $\mathbb R^d$-valued stochastic process such that the  $(d \times d)$-covariation matrix
$[W,Y]$ exists and all the components have finite variation.

The \emph{integral operator $ A^{W,Y}$} is defined on the space $ C_c ([0,T]\times\mathbb R^d;\mathbb R^d)$ by 
   \begin{equation*}
     \begin{array}{lrcl}  
 A^{W,Y}: &  C_c ([0,T]\times\mathbb R^d;\mathbb R^d) &\to & \mathcal C\\
 & l & \mapsto &   A^{W,Y}_\cdot(l) ,
  \end{array}
   \end{equation*}
   where 
   \begin{equation}\label{eq: operator A YW in C} 
   A^{W,Y}_t(l) := \left(\int_0^t l^*(r, W_r) \mathrm d [W,Y]_r \right)^*
     \end{equation}
      for all $t\in[0,T]$. Here $l$  {and $A^{W,Y}_t(l)$} are $d$-dimensional  column vectors.
\end{defin}
{We observe that in the special case when $Y=W$ the occupation time operator $A^{W,W}$ applied to $l$ is nothing but  $\int_0^\cdot l(r, W_r)\mathrm d r$  (see the introduction of Section \ref{ssc: preliminaries AYW} for more details).}
Moreover, for the class of functions $l\in   C_c ([0,T]\times\mathbb R^d;\mathbb R^d)$ the integral in \eqref{eq: operator A YW in C} is well-defined because $[W,Y]$ is a matrix with finite variation components by assumption. Our aim is to define such integral operator $A^{W,Y}$ for generalised functions, as specified in the next definition.

\begin{defin}\label{def: operator A YW in H general}
{Let $E$ be a  Polish space  which 
contains $C_c ([0,T]\times\mathbb R^d;\mathbb R^d)$
as a dense subset.} We define the integral operator $A^{W,Y}: E \rightarrow \mathcal C $ as the continuous extension of the operator defined in Definition \ref{def: operator A YW in C}, provided that the extension exists. 
\end{defin}
In Section \ref{sc: markovian case} we will prove the existence of such extension {for $E =  C([0,T]; H^{s}_{r})$ with $2\leq r<\infty$ and $-\frac12<s\leq 0$.}   Using this extension we can reformulate BSDE \eqref{eq: BSDE classical form} for a rough driver and give a precise meaning to its solution.

\begin{defin}\label{def: solution BSDE integral A}
Let $b \in C([0,T];{\mathcal S}')$.
Let $E$ be a Polish space of {${\mathcal S}'$}-valued functions
 including  $C_c ([0,T]\times\mathbb R^d;\mathbb R^d)$
as a dense subset and such that $b \in E$.
We say that a  continuous $\mathbb R^d$-valued stochastic process $Y$ is a solution of  BSDE \eqref{eq: BSDE classical form} if:
\begin{itemize}
\item[(i)]  $A^{W,Y}$ exists as an operator according to Definition \ref{def: operator A YW in H general};
\item[(ii)] $A_\cdot^{W,Y}(b)$ is a  martingale-orthogonal  process; 
\item[(iii)] $Y_T=\xi$;
\item[(iv)] the process $M= (M_t)_t$ given by
 \begin{equation}\label{eq: BSDE integral A}
M_t:= Y_t -Y_0 +A^{W,Y}_t(b)+\int_0^t f\left (r, W_r, Y_r, \frac{\mathrm d[Y,W]_r}{\mathrm dr}\right) \mathrm dr
 \end{equation} 
 is a  square-integrable $\mathcal F$-martingale, where $\mathcal F$ is the Brownian filtration.
\end{itemize}
\end{defin}

\begin{rem}\begin{itemize}
\item Such solution $Y$ is a weak-Dirichlet process in the sense of Definition \ref{def:weak_dirichlet_process} with  martingale-orthogonal  process $A$ given by $$A^{W,Y}_t(b)+\int_0^t f\left (r, W_r, Y_r, \frac{\mathrm d[Y,W]_r}{\mathrm dr}\right) \mathrm dr  .$$
\item We have $[Y,W]= [M,W]$, thus the covariation process is absolutely continuous with respect to $\mathrm dr$ component by component and hence all terms appearing in the driver $f$ in \eqref{eq: BSDE integral A} are well-defined.
\item Definition \ref{def: solution BSDE integral A} above makes sense also in the case when $\xi$ is a generic square integrable random variable and the random dependence  in the driver $f$ is allowed to be on the whole past $\{W_s; s\leq r\}$ instead of only on $W_r$.
\item   Another  generalization of Pardoux-Peng BSDEs that allows the solution $Y$ not to be a semimartingale appeared for example in \cite{cheriditoNam2017}, where the authors introduce and study generalised backward differential equations. In their formulation of BSDE (see \cite[Definitions 2.1 and 2.2]{cheriditoNam2017}) they consider a functional $F_t(Y,M)$ for every adapted cadlag process $Y$ and $L^p$-martingale $M$,  which in general may not be a semimartingale. In our setting, the object corresponding to this functional  would be $\left(\int_0^t b^*(r, W_r)\mathrm d[Y,W]_r\right)^*$.  However this integral is not defined for every cadlag adapted process $Y$, since $b$ is a distribution and the covariation $[Y,W]$ is not well-defined a priori, hence that setting cannot be used here. 
\end{itemize}
\end{rem}

In the next proposition we see how  the classical  formulation  of the BSDE is  equivalent to the one introduced above if   $b\in   C_c ([0,T]\times\mathbb R^d;\mathbb R^d)$. In   this case clearly $A^{W,Y}$ is itself the trivial extension {to $E=C_c ([0,T]\times\mathbb R^d;\mathbb R^d) $} of the operator introduced in Definition \ref{def: operator A YW in C}.
\begin{prop}\label{pr: equivalence of integral BSDE formulation}
Let $Y$ be a $d$-dimensional adapted process and $b\in   C_c ([0,T]\times\mathbb R^d;\mathbb R^d) $. Then $Y$ is a solution of \eqref{eq: BSDE classical form} according to Definition \ref{def: solution BSDE integral A} with respect to some $E$ if and only if there exists a predictable $(d\times d)$-dimensional process $Z$ such that $(Y,Z)$ is a  solution of BSDE \eqref{eq: BSDE classical form} in the classical sense.
\end{prop}
 
\begin{proof}
Suppose that $(Y,Z)$ is a classical solution  of \eqref{eq: BSDE classical form}. We set $E=  C_c ([0,T]\times\mathbb R^d;\mathbb R^d) $.  This ensures that  point (i) of Definition \ref{def: solution BSDE integral A} is satisfied and $A^{W,Y}_\cdot(b)= \left( \int_0^\cdot  b^*(r, W_r)\mathrm d[W,Y]_r \right)^*$.
Using \eqref{eq: BSDE classical form} we have 
\begin{align*}
&[W,Y]_t \\ 
&= \left[ W , Y_0 -\int_0^\cdot Z_r b(r, W_r)\mathrm d r - \int_0^\cdot f(r, W_r, Y_r, Z_r)  \mathrm d r + \int_0^\cdot Z_r \mathrm d W_r \right]_t,
\end{align*}
where the covariation is a matrix and it is calculated component by component. Clearly the only non-zero term is given by the stochastic integral and so   we get  
\begin{align*}
[W,Y]_t &= \left[W, \int_0^\cdot Z_r \mathrm d W_r   \right] _t 
= \int_0^t  Z_r^* \mathrm d r,
\end{align*}
hence $\mathrm d [W,Y]_r = Z_r^*  \mathrm d r $, and in particular 
$$ A^{W,Y}_\cdot(b)=\left(\int_0^\cdot b^*(r, W_r)Z_r^* \mathrm d r\right)^* = \int_0^\cdot Z_r b(r, W_r) \mathrm d r .$$
 Being of bounded variation, the latter is clearly a martingale-orthogonal  process, which is point (ii) in Definition \ref{def: solution BSDE integral A}. Point (iii) is trivial. Point (iv) is also satisfied because 
\[
Y_t -Y_0 +A^{W,Y}_t(b)+\int_0^t f\left (r, W_r, Y_r, \frac{\mathrm d[Y,W]_r}{\mathrm dr}\right) \mathrm dr = \int_0^t Z_r \mathrm dW_r
\]
and the right-hand side is a square integrable $\mathcal F$-martingale. 

Conversely, let $Y$ be a  solution of \eqref{eq: BSDE classical form} according to Definition \ref{def: solution BSDE integral A} with respect to $E$.
We know that  $$M_t: = Y_t -Y_0 +A^{W,Y}_t(b)+ \int_0^t f (r, W_r, Y_r, \frac{\mathrm d[Y,W]_r}{\mathrm dr})\mathrm dr $$ is a 
square integrable martingale by point (iv) in Definition \ref{def: solution BSDE integral A}, hence by the martingale representation theorem  there exists a square-integrable process $Z$ such that $M_t = \int_0^t Z_r \mathrm d W_r$. Moreover $A^{W,Y}$ is a martingale-orthogonal  process by point (ii), thus $[W,Y]_t  =[W,M]_t = \int_0^t Z^*_r \mathrm d r$. Therefore $ A^{W,Y}_t(b) = ( \int_0^t b^*(s, W_s) \mathrm d [W,Y]_s)^* = \int_0^t Z_s b(s, W_s)\mathrm d s $ and this concludes the proof.
\end{proof}
\begin{rem}
We observe that, in the classical formulation of BSDEs, $Z$ is always directly determined by $Y$ since $\frac{\mathrm d}{\mathrm dt}[Y,W]_t =Z_t$.
\end{rem}

To conclude this section we point out that the new setting and formulation introduced in Definition \ref{def: solution BSDE integral A} in fact coincide with the classical ones even in the case when $b\in L^\infty_{\text{loc}} ([0,T]\times \mathbb R^d; \mathbb R^d)$. This can be seen by observing two facts. The first one is that a BSDE with a driver $b\in L^\infty_{\text{loc}} ([0,T]\times \mathbb R^d; \mathbb R^d)$ makes sense without the introduction of the operator $A$ and  can be studied with classical methods (\emph{\`a la} Pardoux-Peng). 
On the other hand we will show (see Theorem \ref{thm: A(b)=int b for Lp functions}) that  the operator $A^{W,W}$ applied to a driver in $C([0,T]; H^{s}_r)\cap L^\infty_{\text{loc}} ([0,T]\times \mathbb R^d; \mathbb R^d)$ for $-1/2<s\leq 0$ and $2\leq r< \infty$ is compatible with integrals of drivers  in $L^\infty_{\text{loc}} ([0,T]\times \mathbb R^d; \mathbb R^d)$ defined classically.
 Hence the framework presented here coincides   with the classical one   not only for  $b\in   C_c ([0,T]\times\mathbb R^d;\mathbb R^d)$ (as shown in Proposition \ref{pr: equivalence of integral BSDE formulation}) but also  for  $b\in L^\infty_{\text{loc}} ([0,T]\times \mathbb R^d; \mathbb R^d)$.

\section{Analytical PDE results}\label{sc: PDE results}
In  this section we collect and prove some results about several PDEs that will be used in Section \ref{sc: markovian case}. In particular, a key point in the subsequent analysis will be to show that the integral operator $A^{Y, W}$ appearing in \eqref{eq: BSDE integral A} is well-defined for suitable generalised functions and this will be done with the aid of the following auxiliary PDEs and relative results.

The parameters $\beta$ and $q$ are fixed and chosen according to Assumption \ref{ass: parameters}. These are directly linked to the regularity of the rough driver $b$. Moreover the parameters $(\delta,p)$ are chosen in $K(\beta,q)$ and in particular $\frac d \delta < p<q$.  

\vspace{5pt}
The first auxiliary PDE is 
\begin{equation}\label{eq: PDE for phi}
\left\{\begin{array}{l}
\partial_t \phi +\frac12 \Delta \phi = l\\
\phi(T) = \Psi,
\end{array} \right.
\end{equation}
where $\Psi\in H^{1+\delta}_p$ and $l\in C([0,T]; H^{-\beta}_{p})$.  Here the Laplacian $\Delta$ acts on $\phi$ componentwise and the resulting object is a vector with $i$-th component given by $\Delta \phi_i$. With a  slight abuse of notation we  use $\Delta \phi$ for the whole vector. 
We consider the  mild formulation of \eqref{eq: PDE for phi} which is given by
\begin{equation}\label{eq: mild sol phi}
\phi( t) = P(T-t)\Psi + \int_t^T P(r-t) l(r)\mathrm dr,
\end{equation}
where $\{P(t), t\geq0\}$ is the semigroup generated by $\frac12 \Delta$. \\
It is known   that if a classical solution exists then it coincides with the  solution of \eqref{eq: mild sol phi} (mild formulation) and it has certain regularity properties as recalled in the lemma below for smooth $\Psi$ and $l$. For more details and a proof see for example \cite[Theorem 5.1.4, part (iv)]{lunardi95}. 
\begin{lemma}\label{lm: phi is C12}
Let $l\in  C_c^\infty ([0,T]\times\mathbb R^d;\mathbb R^d) $ and  $\Psi\in C^{2+\epsilon}(\mathbb R^d;\mathbb R^d)$ for some $0<\epsilon<1$. The solution $\phi$ to \eqref{eq: PDE for phi} is at least of class $C^{1,2+\epsilon}([0,T]\times \mathbb R^d;\mathbb R^d)$.
\end{lemma}

In the general case that suits our framework (i.e.\ for rough $l$s and $\Psi$ in fractional Sobolev spaces)  we have the following results.
\begin{lemma}\label{lm: continuity of semigroup}
Let $\beta,\delta,p $ and $q$ be chosen according to Assumption \ref{ass: parameters}. 
\begin{itemize}
\item[(i)] If $\Psi \in H_p^{1+\delta}$ then $t\mapsto P(T-t)\Psi$ is a continuous function with values in $ H_p^{1+\delta}$.
\item[(ii)] If  $l\in C\left( [ 0,T]; H^{-\beta}_p \right)  $  then   the function  $ t \mapsto \int_{t}^{T} P(s-t) l\left(  s\right) \mathrm d s$ is in  $C^{\gamma}\left(  \left[  0,T\right]  ;   H^{2-2\epsilon-\beta }_p \right)  $ for every $\epsilon>0$ and $\gamma\in\left(
0,\epsilon\right)  $.\\
In particular one can always choose $\epsilon $ such that $2-2\epsilon-\beta = 1+\delta$.
\end{itemize}
\end{lemma}
\begin{proof}
Item (i) follows from three facts: 1.\ well-known continuity of the heat semigroup $S(t)=e^{-t}P(t)$ in $L^p$; 2.\ continuity of $S(t)$ in $ H_p^{s}$ for all $s\geq0$, which follows from the fact that $A^{s/2}$ commutes with $S(t)$ (see \cite[Chapter 2, Thm 6.13 (b)]{pazy83}) so that one has $\|S(t)w - S(t_0)w\|_{H^s_p} = \|S(t)A^{s/2}w - S(t_0)A^{s/2}w\|_{L^p}$ for each $t_0\in[0,T]$; 3. the link between $S(t)$ and $P(t)$ via the continuous scalar function $e^t$ so that $P(t) = e^t S(t)$ is still continuous in $H^s_p$ as a function of $t$.

Item (ii) follows by first applying  \cite[Proposition 11]{flandoli_et.al14} with  the time $t$ replaced by $T-t$ and then making a change of  time to the resulting integral to get a backward integral, namely transforming the integrator variable $r$ into $s=t-r$.
\end{proof}

\begin{lemma}\label{lm: phi is C01}
Let Assumption \ref{ass: parameters} hold, and let   $\Psi \in H_p^{1+\delta}$ and $l\in C ( [0,T]; H^{-\beta}_{p}) $. The expression $\phi $   given in 
\eqref{eq: mild sol phi}  is well-defined and belongs to $C([0,T];  H_p^{1+\delta})\subset C([0,T];  C^{1+\alpha})$ and to $C^{0,1}$, where $\alpha = \delta-d/p$. Moreover we have 
\[
\| \phi(t)\|_{H_p^{1+\delta}} \leq  \| \Psi\|_{ H_p^{1+\delta}} +  (T-t)^{\frac{1-\delta-\beta}{2}} \|l\|_{C([0,T];H^{-\beta}_{p}) }
\]
and
\[ 
\|\phi\|_{C([0,T]; C^{1+\alpha})}\leq c\|\phi\|_{C([0,T]; H^{1+\delta}_p)} .
\]
\end{lemma}

\begin{proof}
For the first term in \eqref{eq: mild sol phi} we have that $t \mapsto P(T-t) \Psi \in H_p^{1+\delta}$ is continuous by Lemma \ref{lm: continuity of semigroup}, item (i). 
Moreover by  \eqref{eq: P contraction} we have $ \|P(T-t) \Psi \|_{H_p^{1+\delta}}\leq    \| \Psi\|_{ H_p^{1+\delta}}$ for all $t\in[0,T]$. 
For the second term in \eqref{eq: mild sol phi}  we have continuity as a function of time by  Lemma \ref{lm: continuity of semigroup}, item (ii) and again by the  mapping property \eqref{eq: mapping prop Pt} of the semigroup in Sobolev spaces we  get the bound
\begin{align*}  
\left \| \int_t^T P(s-t) l(s)\mathrm ds\right \|_{H_p^{1+\delta}}  & 
\leq c e^T \int_t^T  (s-t)^{-\frac{1+\delta+\beta}{2}} \|l(s)\|_{H^{-\beta}_{p}} \mathrm ds\\
& \leq c (T-t)^{\frac{1-\delta-\beta}{2}} \|l\|_{C([0,T];H^{-\beta}_{p})},
\end{align*}
which ensures that $\phi \in C([0,T];  H_p^{1+\delta})$ since $1-\delta-\beta>0 $ by assumption on the parameters. 
 Moreover  $\delta>d/p$ again by Assumption \ref{ass: parameters} and so by the fractional Morrey inequality (Lemma \ref{lm: fractional Morrey ineq}) we have 
$$ \| \phi(t) \|_{C^{1+\alpha}} \leq c \|\phi (t)\|_{H_p^{1+\delta}}.$$
 Hence taking the supremum over $t\in[0,T]$ we get
$$
\|\phi\|_{C([0,T]; C^{1+\alpha})}\leq c\|\phi\|_{C([0,T]; H^{1+\delta}_p)} .
$$
 From this it follows  that the solution $\phi$ is jointly continuous in $t$ and $x$ and once differentiable in $x$, namely $\phi\in C^{0,1}$ as wanted (for a proof of a result  similar to the last claim see
 \cite[Lemma 21]{flandoli_et.al14}).
\end{proof} 
 The following corollary follows from Lemma \ref{lm: phi is C01} by the linearity of the PDE.
\begin{coroll}\label{cor: convergence of phi_n in C01}
 Let Assumption \ref{ass: parameters} hold. Let  $(l_n)_n \subset C([0,T]; H^{-\beta}_{p})$ be a sequence such that $l_n\to l$ in this space  and let $\Psi_n\to \Psi $ in $H_p^{1+\delta}$ with $(\Psi_n)_n\subset H_p^{1+\delta}$.  Let $\phi_n$    denote the solution of \eqref{eq: PDE for phi} with $l_n$  in place of $l$ and  $\Psi_n$ in place of $ \Psi $. Then $\phi_n\to \phi$ in $C^{0,1}$.
\end{coroll}

Another important PDE that will appear in the next section is the  PDE associated to BSDE \eqref{eq: BSDE classical form} in the Markovian case, which will be used to construct the solution to the BSDE, namely
\begin{equation}\label{eq: PDE for u}
\left\{\begin{array}{l}
\partial_t u(t) +\frac12 \Delta u(t) =- \nabla u^*(t) \, b(t)- f(t,\cdot,u(t),
 \nabla u(t))\\
  u(T) = \Phi.
\end{array}\right. 
\end{equation}
We note that the term  $\Delta u $ (as in PDE \eqref{eq: PDE for phi} above) and the term $\nabla u^*\, b$ are defined componentwise, in particular the $i$-th component of $\nabla u^*\, b$ is given by $\nabla u_i^* b$. A mild solution to PDE \eqref{eq: PDE for u} is a function $u $ that satisfies
\begin{align}\label{eq: mild sol u}
u( t) =& P(T-t) \Phi - \int_t^T P(r-t) \left(\nabla u^*(r) b(r)\right)\mathrm dr \\ \nonumber
& - \int_t^T P(r-t) f(r,\cdot, u(r),\nabla u(r) )\mathrm dr
\end{align}
in an appropriate function space (specified below). Each component in the term $\nabla u^*(r) b(r)$ is defined by means of the pointwise product (recalled in Section \ref{sc: preliminaries}) and it is well-defined as an element of $H^{-\beta}_p$ when $b(t)\in H^{-\beta}_q$ and $\nabla u^*(t) \in H^\delta_p$. 

 Equation \eqref{eq: PDE for u} was first studied   in \cite{issoglio13} on a bounded domain $D\subset\mathbb R^d$ and with $f\equiv 0$. It was then solved in $\mathbb R^d$  in \cite{flandoli_et.al14} with $f=0$, and in \cite{issoglio_jing16} with $f$ non zero. 
{Related non-linear PDEs with rough coefficients have been studied with similar techniques in \cite{hinz_et.al, issoglio19, issoglio_zahle15}. We remark in particular that in \cite{issoglio19} the author applies analytical results on a quadratic rough PDE to study a quadratic rough BSDE. Ideas used there are similar to what has been done in} \cite{issoglio_jing16}, where the authors obtain an existence and uniqueness result 
 for a function $\tilde f:[0,T]\times H^{1+\delta}_p\times H^\delta_p\to H^0_p$ with some Lipschitz regularity and boundedness at 0. We want to apply this result later on, but we will need to consider   $\tilde f$ to be the \emph{same}  function $f$ appearing in BSDE \eqref{eq: BSDE classical form}. Clearly some care is needed because the $f$ appearing in the BSDE is a function of $t, x,y$ and $z$ and its regularity  stated in Assumption \ref{ass: coefficients} is given pointwise, unlike $\tilde f$. On the other hand,  to get a fixpoint  for the PDE we need some Lipschitz regularity in terms of the function spaces. The way to merge these two settings is to consider a function $\tilde f$ (which will have the appropriate Lipschitz regularity) by setting $\tilde f(t,u,v) = f(t, \cdot, u(t), \nabla u(t))$ for any $u\in H^{1+\delta}_p$ and $v\in H^\delta_p$,  with $f$ from Assumption \ref{ass: coefficients} (we will abuse the notation and write $f$ for both).  Then $\tilde f$ satisfies the required conditions, as explained in  \cite[Remark 2.5]{issoglio_jing16}, 
in particular $\tilde f$ is Lipschitz continuous in the Sobolev spaces
\begin{equation}\label{eq: f is lipschitz}
\|\tilde f(t,u,v)-\tilde f(t, u', v')\|_{H^0_p}\leq c(\|u-u'\|_{H^{1+\delta}_p}+\|v-v'\|_{H^{\delta}_p}).
\end{equation}
Theorem 5, and Lemmata 5 and 8 in \cite{issoglio_jing16} give the following existence, uniqueness and regularity result.
 \begin{theorem}[Issoglio, Jing]\label{thm: Issoglio Jing}
Under Assumption \ref{ass: parameters} and Assumption \ref{ass: coefficients} there exists a unique mild solution $u$ to \eqref{eq: PDE for u} in  $C([0,T]; H^{1+\delta}_p)$. Moreover 
  $u(t)\in C^{1+\alpha}$ for all $t\in[0,T]$, where $\alpha= \delta-d/p$, and $u\in C^{0,1}([0,T]\times \mathbb R^d)$.
\end{theorem}
A small note: in \cite{issoglio_jing16} the result  is  valid even if  $b\in L^\infty ([0,T]; H^{-\beta}_{q} )$.

\section{The Markovian case with distributional driver}\label{sc: markovian case}

In this section we  carry out the analysis 
of  BSDE \eqref{eq: BSDE classical form} when $b\in C([0,T]; H^{-\beta}_{q})$ in the \emph{Markovian setting}.   The Markovian setting means here
that the process $Y$ and the r.v.\ $\xi$ are deterministic  functions of $W$, namely $\xi = \Phi(W_T)$ and  $Y_t= \gamma(t, W_t)$ for some deterministic functions $\Phi$ and $\gamma$, the regularity of which is specified below.
 
As already mentioned previously,  one of the main issues when dealing with generalised functions is to show that the integral operator $A^{W,Y}$ can be extended to   
$C([0,T];H^{-\beta}_{q} )$. This extension is performed in Subsection \ref{ssc: preliminaries AYW} below.
In Subsection \ref{ssc: existence} we will show existence (and uniqueness) of a solution  to BSDE \eqref{eq: BSDE classical form} according to Definition \ref{def: solution BSDE integral A} when $b$ is a rough driver.

\subsection{Properties for the occupation time operator $A^{W,W}$}\label{ssc: preliminaries AYW}
In this section we show how to extend the operator $A^{W,Y}$ to generalised functions. 
 Let us focus on the smooth case for a moment. The first key observation is that in the Markovian setting  we can rewrite    $A^{W,Y}$ in terms of the {\em occupation time operator} $A^{W,W}$, where we recall that
\begin{equation}\label{eq: operator A WW in C}
A^{W,W}:  C_c ([0,T]\times\mathbb R^d;\mathbb R^d) \to \mathcal C
\end{equation}
is the integral operator from Definition \ref{def: operator A YW in C} when $Y=W$ and $\mathcal C$ is the space of continuous paths on $[0,T] $ with values in $\mathbb R^d$. 
In the special case when $Y=W$, the covariation is a multiple of the $d$-dimensional identity matrix $ I_d$, so that $\mathrm d[W,W]_r = I_d \mathrm dr$. In particular this means that for any $l\in C_c ([0,T]\times\mathbb R^d;\mathbb R^d)$ we have 
\begin{equation}\label{eq: operator AWW}
A^{W,W}_t(l) = \left( \int_0^t l^*(r, W_r)I_d \mathrm dr\right)^* = \int_0^t l(r, W_r) \mathrm dr,
\end{equation}
for all $t\in[0,T]$. To see that $A^{W,Y}$ can be written in terms of $A^{W,W}$ in the Markovian case, suppose  that there exists a function $\gamma\in C^{0,1}$ such that $Y_t= \gamma(t, W_t)$, hence by Corollary \ref{cor: covariation} we have $[\gamma(\cdot, W) ,  W]_t = \int_0^t \nabla \gamma^*  (r, W_r) \mathrm d r$ and so  $ [W,Y]_t = [W,\gamma(\cdot, W)]_t = \int_0^t\nabla \gamma(r, W_r) \mathrm d r $.
Thus for any  smooth  driver  $l$ in $ C_c ([0,T]\times\mathbb R^d;\mathbb R^d)$
we have the following representation for the integral operator:
\begin{align} \nonumber
A_t^{W,Y}(l) =&  \left( \int_0^t l^*(r, W_r) \mathrm d[W,Y]_r \right)^*\\ 
=& \left(\int_0^t l^*(r, W_r)\nabla \gamma(r, W_r) \mathrm dr \right)^* \nonumber \\
=&  \int_0^t \nabla \gamma^*(r, W_r)  l(r, W_r) \mathrm dr   \nonumber \\
=& A_t^{W,W}(\nabla \gamma^*\, l). \label{eq: AYW = AWW}
\end{align}
By Theorem \ref{thm: Issoglio Jing} $u\in C^{0,1}$ and so equation  \eqref{eq: AYW = AWW} holds true also in the case where $\gamma$ is replaced by the solution $u$ of PDE \eqref{eq: PDE for u}.

\vspace{10pt}
Before going into details on the extension of $A^{W,Y}$ we state a useful density result, the proof of which is postponed to the Appendix.

\begin{lemma}\label{lm: density}
We have $  C_c^\infty ([0,T]\times\mathbb R^d;\mathbb R^d)\subset C([0,T]; H^{s}_{r} ) $ for any $-\frac12<s\leq 0$ and $2\leq r < \infty $, and the inclusion is dense.  
\end{lemma} 

\begin{rem}\label{rm: density}
In  Lemma \ref{lm: density} one can replace  $   C_c^\infty $ with the larger space $ C_c $ and 
therefore obtain that also the space $ C_c ([0,T]\times\mathbb R^d;\mathbb R^d)$ is   dense in  $ C([0,T]; H^{s}_r) $.
\end{rem}
 
The next result provides us with an explicit representation (chain rule) of the occupation time operator $A^{W,W}$ for smooth $l$, and this representation will still hold in the rough case. 
\begin{prop}[Chain rule - smooth case]\label{pr: chain rule in C}
\begin{itemize}
\item[(i)]
Let Assumption \ref{ass: parameters} hold, let $l\in  C_c([0,T]\times\mathbb R^d; \mathbb R^d) $ and  $\Psi\in H^{1+\delta}_p$. Let us denote by  $\phi$ the function given by the  expression
\eqref{eq: mild sol phi}.  
Then  for the integral operator  $A^{W,W}$ given in \eqref{eq: operator A WW in C} we have the representation
\begin{equation}\label{eq: chain rule in C}
A^{W,W}_t(l) = \phi(t, W_t) - \phi(0, W_0)  -  \int_0^t \nabla \phi^* (r, W_r) \mathrm  d W_r,
\end{equation}
for all $t\in[0,T]$.
\item[(ii)] The map on the right-hand side of \eqref{eq: chain rule in C} is continuous with respect to $\phi\in C^{0,1}$.
\end{itemize}
\end{prop}
We note that the structure of the representation \eqref{eq: chain rule in C} does not change when $\Psi$ changes (although obviously the actual function $\phi$ changes when $\Psi$ changes).
\begin{proof}
We prove {\em part (ii)} first. By linearity it is enough to prove it for $\phi=0$. Let $\phi_n\in C^{0,1}$ such that $\phi_n \to 0$ in the same space.  Clearly 
 $\phi_n(\cdot, W)$ converges uniformly to 0 a.s., and in particular uniformly in probability. Setting $f_n = \nabla \phi_n^*$ it remains to show that 
\[
\int_0^\cdot f_n (r, W_r) \mathrm  d W_r \to 0 \text{ u.c.p.}
\]
According to \cite[Proposition 2.26]{karatzasShreve} 
it is enough to show that 
\begin{equation}\label{eq: integral f}
\int_0^T |f_n (r, W_r)|^2 \mathrm  d r \to 0  
\end{equation}
in probability. Now $f_n\to 0 $ uniformly on each  compact by assumption, which implies that \eqref{eq: integral f} holds a.s.

Next we prove {\em part (i)}. Let $(l_n)_n$ be  a sequence in $ C_c^\infty([0,T]\times\mathbb R^d; \mathbb R^d)$   such that $l_n\to l$ in $ C([0,T]; H^{-\beta}_p)$, which can be constructed by Lemma \ref{lm: density} since  $-\frac12<-\beta\leq 0$ and $2\leq p\leq \infty $ by Assumption \ref{ass: parameters}.  
Moreover a similar approximation can be done for $\Psi$, namely since $C^\infty_c$ is dense in  $H^{1+\delta}_p $ (see Step 1 of the proof of Lemma \ref{lm: density}) we can also construct a sequence $(\Psi_n)\subset C^\infty_c$ such that $\Psi_n\to\Psi$ in $H^{1+\delta}_p $.
Let $\phi_n$ denote the expression \eqref{eq: mild sol phi}, 
 where $l$ is replaced by $l_n$ and $\Psi$ by $\Psi_n$.
Then $\phi_n$ is  at least of class $C^{1,2}$ on $[0,T]\times \mathbb R^d$  by Lemma \ref{lm: phi is C12}. 
Given  the expression \eqref{eq: operator AWW}, the PDE \eqref{eq: PDE for phi} and It\^o's formula we get   
\begin{align} \label{E19bis}
A_t^{W,W}(l_n) &= \int_0^t l_n(r, W_r) \mathrm dr  \nonumber\\
&= \int_0^t \left( \partial_t \phi_n(r, W_r) +\frac12 \Delta \phi_n (r, W_r)\right) \mathrm dr\\
&= \phi_n(t, W_t)- \phi_n (0, W_0) - \int_0^t \nabla \phi_n^*(r, W_r)  \mathrm d  W_r, \nonumber
\end{align}
for $0\leq t\leq T$.
By Corollary \ref{cor: convergence of phi_n in C01} we have that $\phi_n\to\phi$ in $C^{0,1}$, thus applying part (ii) we conclude that
\begin{align*} \nonumber
A_\cdot^{W,W} (l) &= \lim_{n\to\infty} A_\cdot^{W,W} (l_n) \\
&= \lim_{n\to\infty} \left( \phi_n(\cdot, W_\cdot) - \phi_n(0, W_0)  -  \int_0^\cdot \nabla \phi_n^* (r, W_r) \mathrm  d W_r\right) \\
 &= \phi(\cdot, W_t) - \phi(0, W_0)  -  \int_0^\cdot \nabla \phi^* (r, W_r) \mathrm  d W_r
\end{align*}
and the proof is complete.
\end{proof}
The following proposition will be used to extend the occupation time  operator $A^{W,W}$ to a suitable space of generalised functions, see Remark \ref{rm: continuity of A WW in H}, part 1. 
\begin{prop}\label{pr: continuity of A WW in C}
 The operator $ A^{W,W}$ (defined in Definition \ref{def: operator A YW in C}  in the special case $Y=W$) is continuous with respect to the topology $C([0,T]; H^{-\beta}_{p})$.
\end{prop}
\begin{proof}
Let $(l_n)_n\subset  C_c ([0,T]\times\mathbb R^d;\mathbb R^d)$ be a sequence such that $ l_n\to 0$  in $ C([0,T]; H^{-\beta}_{p} )$.
Let $\phi_n$ be given  by \eqref{eq: mild sol phi} 
 with $l$ replaced by $l_n$. By Corollary \ref{cor: convergence of phi_n in C01}  we get $\phi_n\to 0$ in $C^{0,1}$. Using  the chain rule (Proposition \ref{pr: chain rule in C} part (i)) and taking the  u.c.p.-limit in $\mathcal C$ as $n\to \infty $ we get by Proposition  \ref{pr: chain rule in C} part (ii)
\begin{align*}
\lim_{n\to \infty} A^{W,W}_\cdot (l_n) 
& =  \lim_{n\to \infty}\left ( \phi_n(\cdot, W) -\phi_n(0, W_0) -  \int_0^\cdot \nabla \phi_n^* (r, W_r) \mathrm  d W_r \right)\\
& =  0.
\end{align*}
The continuity of the occupation time operator $A^{W,W}$ at 0 implies the continuity everywhere by linearity.
\end{proof}

In what follows we are interested in drivers $b\in C([0,T]; H^{-\beta}_{q} ) $, so we would like to extend the operator $A^{W, Y}$ to $b\in C([0,T]; H^{-\beta}_{q} ) $. This will be done by using the occupation time operator $A^{W,W}$, which will be calculated in $\nabla \gamma^*\, b$ for some appropriate function $\gamma$, and $\nabla \gamma^*\, b $ will belong to $C([0,T]; H^{-\beta}_{p} )$. 
For this reason we start by extending the occupation time operator $A^{W,W}$ to the space $E=  C([0,T]; H^{-\beta}_{p} ) $, as explained below.

\begin{rem}\label{rm: continuity of A WW in H} 
\begin{enumerate}
\item By Lemma \ref{lm: density} and Proposition \ref{pr: continuity of A WW in C} we can extend the operator $ A^{W,W} $ continuously to 
$E = C([0,T]; H^{-\beta}_{p} ) $, where the parameters $p$ and $-\beta$ are chosen according to Assumption \ref{ass: parameters}.
So $A^{W,W}$ is well-defined according to Definition \ref{def: operator A YW in H general}. 
\item Clearly the extended operator $A^{W,W}$ defined in Remark \ref{rm: continuity of A WW in H} part 1.\ is continuous, i.e.\ we have 
\[
A^{W,W}_\cdot(l) = \lim_{n\to \infty} A_\cdot^{W,W}(l_n) 
\]
in $\mathcal C$
for any sequence $(l_n)_n$ such that $ l_n\to l$ in $C([0,T]; H^{-\beta}_{p})$.
\end{enumerate}
\end{rem}
We can now  easily prove the chain rule  in the rough case, thus we get an explicit representation of $A^{W,W}_t(l)$ in terms of the solution $\phi$ of equation \eqref{eq: PDE for phi} when  $l\in C([0,T];   H^{-\beta}_{p})$.
\begin{prop}[Chain rule - rough case]\label{pr: chain rule in H}
Let Assumption \ref{ass: parameters} hold,  $l\in C([0,T];   H^{-\beta}_{p})$  and $\phi$ be given by \eqref{eq: mild sol phi} for a  terminal condition $\Psi\in H^{1+\delta}_p$.  Then for all $t\in[0,T]$ we have the representation
\begin{equation}\label{eq: chain rule in H}
A^{W,W}_t(l) = \phi(t, W_t) - \phi(0, W_0)  -  \int_0^t \nabla \phi^* (r, W_r) \mathrm  d W_r.
\end{equation}
Moreover $A^{W,W}(l)$ is a  martingale-orthogonal  process.
\end{prop}
Note that this chain rule does not depend on the actual $ \Psi$ chosen, in particular we can pick $\Psi=0$ or $\Psi=\Phi $.
\begin{proof}
By Lemma \ref{lm: density} we can take a sequence $l_n\to l$ in $C([0,T];   H^{-\beta}_{p})$ such that $(l_n)_n\subset  C_c^\infty ([0,T]\times\mathbb R^d;\mathbb R^d)   $.
 By Remark \ref{rm: continuity of A WW in H} part 2.\ and  the chain rule for the smooth case (Proposition \ref{pr: chain rule in C} part (i)) we get
\begin{align*}
A^{W,W}_\cdot(l)&= \lim_{n\to\infty} A^{W,W}_\cdot (l_n)\\
& =\lim_{n\to\infty} \left( \phi_n(\cdot, W) - \phi_n(0, W_0)  -  \int_0^\cdot \nabla \phi_n^* (r, W_r) \mathrm  d W_r \right).
\end{align*}
Moreover  we can apply Corollary \ref{cor: convergence of phi_n in C01} to $\phi_n$ because indeed  $l_n \to l$ in $ C ([0,T];H^{-\beta}_{p})$ and thus $\phi_n\to\phi$ in $C^{0,1}$. Finally by Proposition \ref{pr: chain rule in C} part (ii)  we can take the   u.c.p.\ limit in $\mathcal C$ when $n \rightarrow \infty$   and   we get
\[ A^{W,W}_\cdot(l) 
 =\phi(\cdot, W) - \phi(0, W_0)  -  \int_0^\cdot \nabla \phi^* (r, W_r) \mathrm  d W_r.\]
To show that $A^{W,W}(l)$ is a  martingale-orthogonal  process we use the representation \eqref{eq: chain rule in H} and calculate the covariation of each term on the right-hand side with  an arbitrary continuous $\mathcal F$-local martingale $N$ with values in $\mathbb R^d$. By Corollary \ref{cor: covariation} 
\[
[\phi(\cdot, W) -\phi(0, W_0), N]_t = [\phi(\cdot, W), N]_t = \int_0^t \nabla \phi^* (r, W_r) \mathrm d[W,N]_r,
\]
having used the fact that $\phi\in C^{0,1} $. Since the covariation operator extends the one of semimartingales,
the covariation of $N$ and the last term on the right-hand side of \eqref{eq: chain rule in H} gives
\[
[-\int_0^\cdot \nabla \phi^* (r, W_r) \mathrm d W_r, N]_t = - \int_0^t \nabla \phi^* (r, W_r) \mathrm d  [W,N]_r,
\]
thus $[A^{W,W}(l), N]_t =0$ as required. 
\end{proof}

\begin{rem}\label{rem:ustunel}
\begin{enumerate}
\item The terminology \emph{occupation time operator} for  $A^{W,W}$ comes from
the extension of the density occupation formula 
$$ \int_0 ^t g(W_s) \mathrm  ds = \int_{\mathbb R} g(a) L_t^W(a) \mathrm da, $$
where $L^W$ is the Brownian local time.
If $g$ is not a function but  $g = h'$, where $h$ is a bounded Borel function,
the extension of the right-hand side is possible  by Bouleau-Yor formula, see \cite{bouleau_yor}. If $X$ is a semimartingale, there the authors
introduce an integral
$\int_{\mathbb R} h(a) L_t^X (\mathrm da)$. Clearly when $X=W$ the integral is well-defined because $L^W$ is itself a semimartingale.
\item In the literature one can find various It\^o type formulae
involving stochastic processes, formally of the type
$\int_0^t g(X_s) \mathrm d[X]_s$ (like in \cite{bouleau_yor} above),  where $X$ is a semimartingale 
and $g$ is a Schwartz distribution. For example in  \cite{russo_vallois95} where $X$ is a (multidimensional) semimartingale and $g = {\rm Hess} f$,
the  integral is formally
expressed as the covariation  $[\nabla f(X),X)]$. In the special case when $X=W$ is a Brownian motion (so $[W]_t \equiv t$)  those papers expanded $f(W_t)$ for some $f \in C^1(\mathbb R)$ (resp. $f \in C^1(\mathbb R^d)$) and
$g$ is the distribution $\Delta f$. In particular
those formulae focused on the pointwise composition $f(X)$.
\item Using a different approach, \cite{ustunel1982} expanded abstractly $T \star \delta_{W_t}$, where $T$ is a Schwartz distribution and 
$W$ is a standard Brownian motion.
 Taking $T$ associated with a   
$C^1$ function $f$, this would imply the expansion 
of the function $x_0 \mapsto f(W_t + x_0)$. By an easy adaptation of It\^o's formula shown in \cite{ustunel1982} one gets $\mathrm dx_0$-a.e.\
\begin{equation*}
  f(W_t + x_0) =  f(W_0 + x_0) + \int_0^t \nabla f(W_s + x_0) \mathrm dX_s 
+ \mathcal A_t(\Delta f) (x_0),
\end{equation*}
where $x_0  \mapsto \mathcal A_t(\Delta f) (x_0)$ is (for each $t$)
a random field a.s.\ associated with the random distribution
\begin{equation}\label{eq:phi}
\varphi \mapsto   \int_{\mathbb R^d}  \mathcal A_t(\Delta f)(x_0) \varphi(x_0) \mathrm dx_0 = 
\int_0^t (\Delta f \star \varphi)
(W_s + x_0) \mathrm ds  .
\end{equation}
\item This can be linked to the occupation time operator, indeed the right-hand side of \eqref{eq:phi} can be seen as $ A^{W,W}_t (\Delta f \star \varphi(x_0 + \cdot))$. 
By continuity with respect to $x_0$ it is possible
to extend It\^o's formula to every $x_0$.
At this point, if we formally take $\varphi = \delta_{x_0}$,
 then we recover the chain rule stated in Proposition \ref{pr: chain rule in H} in the special case where $f$ is time-independent. The rigorous proof however, 
would need mollifications of $\delta_{x_0}$ and a limiting procedure, which in essence is the same idea we used (translated in our context) when we defined the extended operator $A^{W,W}$. 
\end{enumerate}
\end{rem}

The next lemma is a continuity result that will be used in Proposition \ref{pr: continuity of A YW in C} to show the extension of the operator $A^{W,Y}$ to $C([0,T];   H^{-\beta}_{q})$.

\begin{lemma}\label{lm: continuity of pointwise product}
Let $\gamma \in C([0,T]; H_p^{1+\delta})$. For any sequence $(l_n)_n\subset C([0,T];   H^{-\beta}_{q})$	 such that $l_n\to l$ in $C([0,T];   H^{-\beta}_{q})$, then $ \nabla\gamma^*\, l $ is an element of $C([0,T];   H^{-\beta}_p) $ and 
 $\nabla \gamma^*\, l_n\to \nabla\gamma^*\, l$ in the same space.
\end{lemma}
\begin{proof}
 In the space $ H^{-\beta}_p$ the norm of the pointwise product  for each $t$
 $$
\| \nabla\gamma^*(t) l_n(t) - \nabla\gamma^*(t) l(t)\|_{ H^{-\beta}_p} =  \| \nabla \gamma^*(t) (l_n(t)-l(t))\|_{ H^{-\beta}_p}
 $$ 
  is bounded by  $c \|\nabla \gamma^*(t)\|_{ H^{1+\delta}_p} \| l_n(t)-l(t) \|_{  H^{-\beta}_q}$ thanks to Lemma  \ref{lm: pointwise product} applied to each component.
  Taking the supremum over time  $t\in[0,T]$ we get 
\[
\sup_{t\in[0,T]}\|\nabla\gamma^*(t) (l_n - l)(t)\|_{ H^{-\beta}_p} \leq c \|\gamma^*\|_{C([0,T]; H^{1+\delta}_p)} \| l_n-l \|_{C([0,T]; H^{-\beta}_q)}
\]
and the right-hand side goes to zero as $n\to \infty$ by assumption. This concludes the proof.
\end{proof}

\begin{prop}\label{pr: continuity of A YW in C}
Let Assumption \ref{ass: parameters} hold.  Suppose 
 $Y_t=\gamma(t, W_t)$ for some $\gamma\in C([0,T]; H_p^{1+\delta})$. Then the
  map $ A^{W,Y}$ is well-defined in the sense of Definition
\ref{def: operator A YW in H general} with $E= C([0,T]; H_q^{-\beta})$ and
\begin{equation} \label{E511}
A^{W,Y}(l)=A^{W,W}(\nabla\gamma^*\, l), 
\end{equation}
for all $l \in E$.
\end{prop}
\begin{proof}
We start by observing that $C_c ([0,T]\times\mathbb R^d;\mathbb R^d)$ is dense in 
$ E = C([0,T]; H_q^{-\beta})$  by Lemma \ref{lm: density}. Moreover  $A^{W,W} $ is well-defined   in $ C([0,T]; H_p^{-\beta})$ by Remark \ref{rm: continuity of A WW in H} part 1.\ and it is continuous. Let $l_n \to l$ in $E$. 
We want to prove that $A^{W,Y}(l_n)$ converges to the RHS of \eqref{E511}.
Taking into account \eqref{eq: AYW = AWW} and the fact that $l_n \in C_c ([0,T]\times\mathbb R^d;\mathbb R^d)$ we have 
\[
A^{W,Y}(l_n)=A^{W,W}(\nabla\gamma^*\, l_n).
\]   
 Note  that   the map $l \mapsto \nabla \gamma^* \, l$ is continuous from $ C([0,T]; H^{-\beta}_{q})$ to $ C([0,T]; H^{-\beta}_{p})$ thus   $A^{W,W}(\nabla \gamma^* \, l_n ) \to A^{W,W}(\nabla \gamma^* \, l ) $ in $\mathcal C$ because of compositions of continuous maps. This concludes the proof. 
\end{proof}
 
 \begin{rem}
 We observe that in \cite{issoglio_jing16} the authors deal with the singular integral term $\int_0^t Z_s b(s,W_s)\mathrm ds $ by replacing it with known terms. In particular, they \emph{define it using the chain rule} \eqref{eq: chain rule in H} with $l= \nabla u^*\, b$ but without proving it. Their virtual solution coincide with the one constructed here.
 \end{rem}
Finally we end this section with a result on classical drivers $g$. We show that for a function $g\in C([0,T]; H^{s}_r)\cap L^\infty_{\text{loc}}([0,T]\times \mathbb R^d; \mathbb R^d)$ with $-\frac12<s\leq0$ and $2\leq r< \infty$,  then the operator $A^{W,W}$ defined as an extension to $E=C([0,T]; H^{s}_{r} )$ and evaluated in  $g$ coincides with the classical integral $\int_0^\cdot g(s, W_s) \mathrm ds$. The proof of the theorem below is postponed to the Appendix for ease of reading.
\begin{theorem}\label{thm: A(b)=int b for Lp functions}
Let $g\in C([0,T]; H^{s}_r)\cap L^\infty_{\text{loc}}([0,T]\times \mathbb R^d; \mathbb R^d)$ with $-\frac12<s\leq 0$ and
 $2\leq r < \infty$, with $g$ column vector. Suppose that $A^{W,W}$ is well-defined in the sense of Definition \ref{def: operator A YW in H general} with $E= C([0,T]; H^s_r)$. Then 
\begin{equation}\label{eq: A(b)=int b for Lp functions}
A_\cdot^{W,W}(g) = \int_0^\cdot g(s, W_s) \mathrm ds.
\end{equation}
\end{theorem}
Note that the operator $A^{W,W}$ is well-defined for example if $s=-\beta$ and $r=p$ see Remark \ref{rm: continuity of A WW in H}.

\begin{coroll}[chain rule for $L^\infty_{\text{loc}}$]
If $g\in L^\infty_{\text{loc}}([0,T]\times \mathbb R^d;\mathbb R^d)\cap C([0,T]; H^{-\beta}_p)$ then
\[
\int_0^tg(s, W_s) \mathrm ds = \phi(t, W_t) - \phi(0, W_0) - \int_0^t \nabla \phi^* (s, W_s) \mathrm d W_r,
\]
where $\phi$ is the solution of \eqref{eq: PDE for phi} with $\Psi\in H^{1+\delta}_p$, given by \eqref{eq: mild sol phi}.
\end{coroll}	
\begin{proof}
This follows by Theorem \ref{thm: A(b)=int b for Lp functions} and Proposition \ref{pr: chain rule in H} with $l=g$.
\end{proof}

\subsection{Existence for the BSDE and  Feynman-Kac representation}\label{ssc: existence}

Here we 
show that the solution of PDE  \eqref{eq: PDE for u} can be used to construct a solution to  BSDE 
 \eqref{eq: BSDE classical form} when $b\in C([0,T]; H^{-\beta}_{q})$.  In particular, in Theorem \ref{thm: Markovian BSDE - existence of sol} we construct a solution to BSDE  \eqref{eq: BSDE classical form} with $\xi = \Phi(W_T)$ using the solution to the associated PDE. As a corollary of Theorem \ref{thm: Markovian BSDE - existence of sol} we obtain a Feynman-Kac representation in Corollay \ref{cor: Feynman-Kac repr}.

For ease of reading, we rewrite the formal meaning of the 
BSDE \eqref{eq: BSDE classical form} under consideration:
\begin{equation*}
Y_t = \Phi(W_T)  +\int_t^T Z_r b(r, W_r) \mathrm d r  + \int_t^T f(r, W_r, Y_r, Z_r)  \mathrm d r -  \int_t^T Z_r \mathrm d W_r.
 \end{equation*}

\begin{theorem}\label{thm: Markovian BSDE - existence of sol}
Let Assumption \ref{ass: parameters} and Assumption \ref{ass: coefficients} hold and let $b\in C([0,T]; H^{-\beta}_{q})$.
We denote by $u$ be the unique mild solution to \eqref{eq: PDE for u}.  Then $Y_t= u(t, W_t)$ is a solution of \eqref{eq: BSDE classical form} according to Definition \ref{def: solution BSDE integral A} with $E= C([0,T]; H^{-\beta}_{q})$.
\end{theorem}

\begin{proof}
First we observe that thanks to Theorem \ref{thm: Issoglio Jing} we have $u\in C([0,T]; H^{1+\delta}_p)$. Thus by Proposition \ref{pr: continuity of A YW in C} the operator  $A^{Y,W} $ appearing in Definition \ref{def: solution BSDE integral A}   is well-defined in $E= C([0,T]; H^{-\beta}_{q})$ and
we have  
\begin{equation} \label{EF518}
A^{W,Y}_t(b)  =  A^{W,W}_t(\nabla u^*\, b).
\end{equation}
 This is a  martingale-orthogonal  process by Proposition 
\ref{pr: chain rule in H} with $l= \nabla u^*\,b$. The latter is an element of $  C([0,T]; H^{-\beta}_{p})$, and this is shown by  Lemma \ref{lm: continuity of pointwise product}.
Moreover $u(T)= \Phi$ implies that $ Y_T = u(T, W_T)= \Phi( W_T) $ so that
parts (i)-(iii) of
 Definition \ref{def: solution BSDE integral A} are verified.
The last point to check is part (iv) in the same Definition, namely that
 $$M_t:= Y_t -Y_0 +A^{W,Y}_t(b)  +\int_0^t f\left (r, W_r, Y_r, \frac{\mathrm d[Y,W]_r}{\mathrm dr}\right) \mathrm dr$$ is a square integrable martingale.   The term with the driver $f$ becomes
\begin{align} \label{EEE}
\int_0^t f\left (r, W_r, Y_r, \frac{\mathrm d[Y,W]_r}{\mathrm dr}\right) \mathrm dr &= \int_0^t f\left (r, W_r, u(r, W_r),\nabla u (r, W_r) \right) \mathrm dr \nonumber \\
&= \int_0^t \tilde f(r, W_r) \mathrm dr,
\end{align}
where $\tilde f(t,x)= f(t, x, u(t,x), \nabla u(t,x))$.  Since  $u \in C^{0,1}$ and $f$ is continuous then $\tilde f\in L^\infty_{\text{loc}}([0,T]\times \mathbb R^d)$.
We also have that $\tilde f \in C([0,T]; L^p)$ since $f$ is Lipschitz in $(y,z)$ uniformly in $t, x$, and $x\mapsto f(t,x,0,0)$ is an element of $L^p$ uniformly in $t\in[0,T]$ by Assumption \ref{ass: coefficients} and $u(t), \nabla u(t)$ are in $L^p$ uniformly in $t$ since $u \in C([0,T]; H^{1+\delta}_p)$. 
 So in particular $\tilde f \in C([0,T]; H^{0}_p)$ and hence by Theorem \ref{thm: A(b)=int b for Lp functions} we have  
\[
\int_0^t \tilde f(r, W_r) \mathrm dr = A^{W,W}_t(\tilde f).
\]
Moreover  by \eqref{EF518} and  the linearity of $A^{W,W}$ one gets
\begin{align*}
 M_t &=  Y_t -Y_0 +A^{W,W}_t(\nabla u^*\, b ) + A^{W,W}_t(\tilde f)\\
 &=  Y_t -Y_0 +A^{W,W}_t(\nabla u^*\, b + \tilde f)\\
 &=  Y_t -Y_0 -A^{W,W}_t(-\nabla u^*\, b - \tilde f).
\end{align*}
 Now we apply the chain rule  to $A^{W,W}_t(-\nabla u^*\, b - \tilde f)$, namely Proposition \ref{pr: chain rule in H} with $l=-\nabla u^*\, b - \tilde f$ on the RHS of \eqref{eq: PDE for phi}.
Note that in this case \eqref{eq: chain rule in H} holds for $\phi = u$ because the function $u$ verifies \eqref{eq: mild sol phi} with $l=-\nabla u^*\, b - \tilde f $, see indeed \eqref{eq: mild sol u}. Thus we get 
\begin{align*}
M_t = & Y_t -Y_0 -A^{W,W}_t(-\nabla u^*\, b - \tilde f)\\
= & u(t, W_t) - u(0, W_0) \\ 
&-  u(t, W_t) + u(0, W_0)  +  \int_0^t \nabla u^* (r, W_r) \mathrm  d W_r 
\end{align*}
so that
\[
M_t   =  \int_0^t \nabla u^*(r, W_r) \mathrm  d W_r,
\]
which is clearly a  square integrable   $\mathcal F$-martingale  because $\nabla u^*$ is uniformly bounded since $u\in C^{1+\alpha}$ by Theorem \ref{thm: Issoglio Jing}.
\end{proof}

\begin{coroll}\label{cor: Feynman-Kac repr}
Under the hypothesis of Theorem \ref{thm: Markovian BSDE - existence of sol} we have the Feynman-Kac (implicit) representation for the solution $u$ of PDE \eqref{eq: PDE for u} given by
\begin{align*}
u(s,x_0) &= \mathbb E \bigg[ \Phi( x_0+W_{T -s}) \\
&+ \int_s^T f( r,W_r +x_0, u(r, W_r+x_0), \nabla u (r, W_r +x_0) ) \mathrm dr \\
&+ A_T^{W ,W } ((\nabla u^* b)(x_0 + \cdot ) ) -A_s^{W ,W } ((\nabla u^* b)(x_0 + \cdot ) )  \bigg]
\end{align*}
for all $s\in[0,T]$ and $x_0\in \mathbb R^d$.
\end{coroll}

\begin{proof}
For ease of proof we show the result for $s=0$. We set $\hat f(t,x,y,z) := f(t,x+x_0,y,z), \hat\Phi(x):= \Phi(x+x_0)$ and  (formally) $\hat b(t,x):= b(t, x+x_0)$.
Let $\hat u$ be the solution of PDE \eqref{eq: PDE for u} where the coefficients $b,f,\Phi$ are replaced by $\hat b, \hat f$ and $\hat \Phi$. It is easy to see that $\hat u(t,x)= u(t,x+x_0)$,
where $u$ is the solution to the original PDE.

If we now consider BSDE  \eqref{eq: BSDE classical form}    where the coefficients $b,f,\Phi$ are replaced by $\hat b, \hat f$ and $\hat \Phi$, then by Theorem \ref{thm: Markovian BSDE - existence of sol} we know that    $Y_t = \hat u(t, W_t) = u(t, x_0+ W_t)$ is a solution according to Definition  \ref{def: solution BSDE integral A}. In particular we have for all $t\in[0,T]$
\begin{align*}
Y_t = & \hat \Phi( W_T ) + \int_t^T \hat f\left( r,W_r , Y_r ,  \frac{[Y,W]_r}{\mathrm dr}\right) \mathrm dr \\
& + A_{T}^{W ,Y } ( \hat  b) -A_{t}^{W ,Y } ( \hat  b) -( M_T -M_t ),                                                                                                          
\end{align*}
where $M_t$ is an $\mathcal F_t$-martingale. We now use the explicit expression of $Y$ in terms of $u$ and Corollary \ref{cor: covariation} to replace the bracket, and taking the expectation we get for $t=0$
\begin{align*}
u(0,x_0) = & \mathbb E\bigg[ \Phi(W_T+x_0) \\
& + \int_0^T f\left ( r,W_r +x_0, u(r, W_r+x_0), \nabla u(r, W_r+x_0){\mathrm dr}\right) \mathrm dr \\
& + A_T^{W ,W } ( (\nabla u^* b)(x_0 + \cdot ) )\bigg],
\end{align*}
having used Proposition \ref{pr: continuity of A YW in C} in the last step to replace $A^{Y,W}$ with $A^{W,W}$.
\end{proof}

We conclude with a result about the uniqueness of the solution $Y$ in the class $Y_t=\gamma(t, W_t)$ for certain $\gamma$s.
\begin{prop}\label{prop: Markovian BSDE - uniq of sol}
Let Assumption \ref{ass: parameters} and Assumption \ref{ass: coefficients} hold and let $b\in C([0,T]; H^{-\beta}_{q})$. If the solution of \eqref{eq: BSDE classical form} according to Definition \ref{def: solution BSDE integral A} with $E=C([0,T]; H^{-\beta}_{q})$ can be written as $Y_t=\gamma(t, W_t)$ for some $\gamma\in C([0,T]; H^{1+\delta}_p)$, then it is unique. 
\end{prop}
 
\begin{proof}
Suppose that $Y^i_t=\gamma^i(t, W_t)$, $i=1,2$ are solutions to \eqref{eq: BSDE classical form} according to Definition \ref{def: solution BSDE integral A} 
and let us denote by
\begin{equation}\label{eq: Mt i}  
M^i_t:=Y^i_t-Y^i_0+A_t^{W,Y^i}(b)+\int_0^t f\left (r, W_r, Y^i_r, \frac{\mathrm d[Y^i,W]_r}{\mathrm dr}\right) \mathrm dr, 
\end{equation}
which is a martingale by part (iv) of Definition \ref{def: solution BSDE integral A}. Moreover \begin{equation}\label{eq: l}
 (\nabla \gamma^i)^*\, b \in C([0,T]; H^{-\beta}_p) 
 \end{equation}
  by Lemma \ref{lm: continuity of pointwise product}.
By assumption on $Y^i$ we can apply Proposition \ref{pr: continuity of A YW in C} and write 
\begin{equation} \label{EAYW}
A_t^{Y^i,W}(b) = A_t^{W,W}((\nabla \gamma^i)^* \,b).
\end{equation}
 Furthermore by Corollary \ref{cor: covariation}
we have 
\begin{align} \label{T516}
&\int_0^t f\left (r, W_r, Y^i_r, \frac{\mathrm d[Y^i,W]_r}{\mathrm dr}\right) \mathrm dr \nonumber\\
&= \int_0^t f\left (r, W_r, \gamma^i(r,W_r), \nabla \gamma^i(r,W_r) \right) \mathrm dr \nonumber\\
&= \int_0^t \tilde f^i\left (r, W_r \right) \mathrm dr  ,
\end{align}
where $\tilde f^i (t, x):= f(t, x, \gamma^i(t,x), \nabla \gamma^i(t,x))$. We note that 
\begin{equation}\label{eq: f tilde}
\tilde f^i \in  L^\infty_{\text{loc}} \cap C([0,T]; L^p),
\end{equation} 
which can be proven similarly to the considerations below \eqref{EEE} in the proof of the previous existence theorem. Thus we can apply Theorem \ref{thm: A(b)=int b for Lp functions}, so 
$\eqref{T516} =  A_t^{W,W}(\tilde f^i) $ . 
By \eqref{EAYW}  and the additivity of $A^{W,W}$  we have 
\begin{equation}\label{eq: martingale Mi}
M^i_t=Y^i_t-Y^i_0+A_t^{W,W}((\nabla \gamma^i)^*\, b + \tilde f^i).
\end{equation} 
Let us consider the PDE 
\begin{equation}\label{eq: mild solution h with h(T)=0}
\left\{\begin{array}{l}
\partial_t h^i(t) +\frac12 \Delta h^i (t)= (\nabla \gamma^i)^*(t)\, b (t)+ f(t, \cdot,  \gamma^i, \nabla \gamma^i) \\
h^i(T) = 0,
\end{array} \right.
\end{equation}
which is PDE \eqref{eq: PDE for phi} with
$ (\nabla \gamma^i)^*(t)\, b(t)+ f(t, \cdot,\gamma^i, \nabla \gamma^i)= (\nabla\gamma^i)^*(t)\, b (t)+ \tilde f^i(t, \cdot) \in C([0,T]; H^{-\beta}_p)$ (by \eqref{eq: f tilde} and \eqref{eq: l})  on the right-hand side in place of $l$.   We denote by $h^i$, $i=1,2$  the corresponding (mild solution)  expression \eqref{eq: mild sol phi}, which belongs to $C([0,T]; C^{1+\alpha})$ by  Lemma \ref{lm: phi is C01}. Then $(\nabla h^i)^*$ is bounded.
By the   chain rule (Proposition \eqref{pr: chain rule in H}) we get
\begin{equation}\label{eq: chain rule applied to gamma i}
 A^{W,W}_t((\nabla \gamma^i)^*\, b + \tilde f^i) = h^i(t,W_t) -h^i(0,W_0)- \int_0^t(\nabla h^i )^*(r, W_r) \mathrm d W_r.
\end{equation}
Plugging \eqref{eq: chain rule applied to gamma i} into \eqref{eq: martingale Mi}  we get 
\begin{align*}
M^i_t=&\gamma^i(t, W_t)-\gamma^i(0, W_0)+ h^i(t, W_t)-h^i(0, W_0)\\
& -\int_0^t (\nabla h^i)^* (r, W_r) \mathrm d W_r.
\end{align*}
Subtracting $M_T^i$ from both sides and rearranging the terms  we obtain 
\begin{align}\label{eq: gamma i=acca i}
\nonumber \gamma^i(t, W_t)+h^i(t, W_t)  =& - (M^i_T -M^i_t) -\int_t^T (\nabla h^i)^* (r, W_r) \mathrm d W_r\\
&+ \gamma^i(T, W_T)+h^i(T, W_T)\\
=& \Phi(W_T)  - (\tilde M^i_T -\tilde M^i_t), \nonumber
\end{align}
where we have set 
$\tilde M^i_t: = M^i_t+ \int_0^t \nabla h^i (r, W_r) \mathrm d W_r  $ and we have used the fact that $h^i(T, W_T) = 0$ by \eqref{eq: mild solution h with h(T)=0} and that $\gamma^i(T, W_T) = \Phi(W_T)$ by item (iii) of Definition \ref{def: solution BSDE integral A}. Clearly $\tilde M^i$  is  another martingale since $(\nabla h^i)^*$ is bounded.
So the left-hand side of equality \eqref{eq: gamma i=acca i} can be represented by
\[
\gamma^i(t, W_t)  + h^i(t, W_t) = E\left[ \Phi(W_T)\vert \mathcal F_t\right].
\]
The above equality holds for $i=1,2$ and since the right-hand side is the same, we get 
 \[
\gamma^1(t, W_t) + h^1(t, W_t)  = \gamma^2(t, W_t) + h^2(t, W_t)  
 \]
 almost surely. From this we can infer that  
 \begin{equation}\label{eq: gamma + h}
\gamma^1(t, x) + h^1(t, x)  = \gamma^2(t, x) + h^2(t, x) , 
 \end{equation}
 for   every $t\in[0,T] $ and $x\in\mathbb R^d$ in the following way: suppose that we have a continuous function $\eta$ such that $\eta(t,W_t)=0$ almost surely.  Then 
\[
0=E[ \vert \eta(t,W_t)\vert ]  = \int_{[0,T]\times \mathbb R^d} \vert \eta(t,x)\vert
  p_t(x) 
\mathrm dt \mathrm dx
\] 
and since $p_t(x)>0$ we get that $\eta(t,x)=0$ almost everywhere. 
  In fact this holds everywhere because $\eta$ is continuous. Setting $\gamma^i(t):= \gamma^i(t, \cdot)$ for all $t\in[0,T]$ and $i=1,2$, it remains to show that $\gamma^1=\gamma^2$.  
 We know that  $\gamma^1(t)-\gamma^2(t) =  h^2(t)-h^1(t)$ by \eqref{eq: gamma + h}. The idea is to bound the difference $h^2-h^1 $ in the $\rho$-equivalent norm for the space ${C([0,T]; H^{1+\delta}_p)}$. To do so we work with the time reversed functions $\hat h^i (s) := h^i(T-s)$, which clearly have the same regularity as $h^i$ and also the same norm in $C([0,T]; H^{1+\delta}_p)$ and $\rho$-equivalent norm. Setting $\hat b(s): = b(T-s), \hat \gamma^i(s):= \gamma^i(T-s)$ and $\hat f (s,y,z) := f(T-s, y,z)$ we have
\begin{align*}
\hat  h^2(t)-&\hat h^1(t) 
= \int_0^t P({t-r})  \left( ( \nabla \hat \gamma^2(r)- \nabla \hat \gamma^1(r) )^* \hat b(r) \right) \mathrm dr \\
& + \int_0^t P({t-r}) \left ( \hat f(r,  \hat \gamma^2(r), \nabla\hat  \gamma^2(r) -\hat f( r, \hat \gamma^1(r), \nabla \hat \gamma^1(r) ) \right)  \mathrm dr. 
\end{align*}
Taking the $\rho$-equivalent norm (see \eqref{eq: rho equiv norm}) of the difference above, we have
\begin{align*}
&\|\hat  h^2-\hat h^1 \|_{C([0,T]; H^{1+\delta}_p)}^{(\rho)}\\
 & = \sup_{0\leq t\leq T} e^{-\rho t}\|\hat  h^2(t)-\hat h^1(t) \|_{H^{1+\delta}_p}\\
&\leq \sup_{0\leq t\leq T} e^{-\rho t} \|
\int_0^t P({t-r})  \left( ( \nabla \hat \gamma^2(r)- \nabla \hat \gamma^1(r) )^* \hat b(r) \right) \mathrm dr \|_{H^{1+\delta}_p}\\
&\phantom{eeeee} + \sup_{0\leq t\leq T} e^{-\rho t} \|
\int_0^t P({t-r})  \Bigg( \hat f(r,  \hat \gamma^2(r), \nabla\hat  \gamma^2(r) \\
&\phantom{eeeee} -\hat f( r, \hat \gamma^1(r), \nabla \hat \gamma^1(r) ) \Bigg)  \mathrm dr \|_{H^{1+\delta}_p}\\
&=: (A) + (B).
\end{align*} 
To bound the first term we use the pointwise product estimate for fixed time $r\in[0,T]$ (Lemma \ref{lm: pointwise product}), the mapping property \eqref{eq: mapping prop Pt} of the semigroup, and the definition of the $\rho$-equivalent norm \eqref{eq: rho equiv norm}. We get
\begin{align} \label{EA}
(A)&\leq c \sup_{0\leq t\leq T} \int_0^t e^{-\rho t} (t-r)^{-\frac{1+\delta+\beta}{2}}  \|\hat b(r)\|_{H^{-\beta}_q} \|\nabla \hat \gamma^2(r) -\nabla \hat \gamma^1(r)\|_{H^{\delta}_p} \mathrm dr \nonumber \\
&\leq c\|\hat b\|_{C([0,T];H^{-\beta}_q)} \sup_{0\leq t\leq T} \int_0^t e^{-\rho (t-r)} (t-r)^{-\frac{1+\delta+\beta}{2}} \cdot  \nonumber \\
&\phantom{dddddddddddddddddddddddd} e^{-\rho r}  \| \hat \gamma^2(r) - \hat \gamma^1(r)\|_{H^{1+\delta}_p} \mathrm dr \\
&\leq c \| \hat \gamma^2 - \hat \gamma^1\|_{C([0,T];H^{1+\delta}_p)}^{(\rho)} \sup_{0\leq t\leq T} \int_0^t e^{-\rho (t-r)} (t-r)^{-\frac{1+\delta+\beta}{2}} \mathrm dr \nonumber\\
& \leq c \rho^\frac{\delta+\beta-1}{2} \| \hat \gamma^2 - \hat \gamma^1\|_{C([0,T];H^{1+\delta}_p)}^{(\rho)},\nonumber
\end{align}
having used  the Gamma function and the bound  $$\int_0^te^{-\rho r} r^\alpha \mathrm dr\leq  \Gamma(\alpha+1) \rho^{-(\alpha+1)} $$ in the latter inequality, with $\alpha ={-\frac{1+\delta+\beta}{2}}$. Note that $-(\alpha+1)= \frac{\delta+\beta-1}{2}<0 $ so we have $\rho^\frac{\delta+\beta-1}{2} \to 0$ as $\rho\to \infty$.\\
To bound term (B) we do similarly but use the mapping property of the semigroup from $H^{0}_p$ to $H^{1+\delta}_p$ and the Lipschitz regularity \eqref{eq: f is lipschitz} of $\hat f$ so we get
\begin{align} \label{EB}
(B)& \leq c \sup_{0\leq t\leq T}  
\int_0^t e^{-\rho (t-r)}  (t-r)^{-\frac{1+\delta}{2}}\cdot  \nonumber\\
&\phantom{ddddddddddd} e^{-\rho r} (c \| \hat \gamma^2(r) - \hat \gamma^1(r) \|_{H^{1+\delta}_p}  + \|\nabla \hat \gamma^2(r) -\nabla \hat \gamma^1(r) \|_{H^{\delta}_p} ) \mathrm dr \\
&\leq c \rho^{\frac{\delta-1}{2}} \|\hat \gamma^2 - \hat \gamma^1\|_{C([0,T];H^{1+\delta}_p)}^{(\rho)}. \nonumber
\end{align}
Collecting the estimates \eqref{EA} and \eqref{EB}, we get 
\[
\|\gamma^1-\gamma^2\|^{(\rho)}_{C([0,T]; H^{1+\delta}_p)} \leq c(\rho^\frac{\delta-1}{2} + \rho^\frac{\delta+\beta-1}{2} ) \|\gamma^1-\gamma^2\|^{(\rho)}_{C([0,T]; H^{1+\delta}_p)},
\]
so 
\begin{equation*} 
\|\gamma^1-\gamma^2\|^{(\rho)}_{C([0,T]; H^{1+\delta}_p)} ( 1-c(\rho^\frac{\delta-1}{2} +  \rho^\frac{\delta+\beta-1}{2})) \leq 0,
\end{equation*}
where $c$ depends on $b$ and $T$ but not on $\gamma^i$ or $\rho$. We choose   $\rho$   large enough such that $1-c(\rho^\frac{\delta-1}{2} + \rho^\frac{\delta+\beta-1}{2})>0$,  
which implies $\gamma^1=\gamma^2$ and shows shows that $Y^1 = Y^2$. 
\end{proof}

Note that, if one wanted to generalize 
this framework to the non-Markovian case, one possibility would be to use functional It\^o calculus. 
This approach however, would need an analytical study of a path-dependent PDE like \eqref{eq: PDE intro}, and the current analytical tools seem insufficient at the moment.

\appendix
\section{A technical lemma and proofs of  Lemma \ref{lm: density} and Theorem \ref{thm: A(b)=int b for Lp functions}}\label{sc: appendix}
We first state and prove a technical Lemma that is used in the proofs below. 
\begin{lemma}\label{lm: projection converge uniformly in compacts}
Let $(H, \|\cdot\|)$ be a normed space and $(P_N)_N$ be a family of linear equibounded operators on $H$ such that for each $a\in H$ we have $P_Na\to a$ in $H$. Then for any compact $K\subset H$ we have
\[
\sup_{a\in K}\| P_N a- a\|\to 0,
\]
as $N\to\infty$.
\end{lemma}
\begin{proof}
Let $\delta >0$. Since $K$ is compact, we can construct a finite cover   of size $\delta$, for example $K\subseteq \cup_{i=1}^m B(a_i, \delta)$. For a given $a\in H$ there exists $j\in\{1, \ldots, m\}$ such that $a\in B(a_{j}, \delta)$. Then we write
\begin{align*}
\|P_N a -a \| &\leq \|P_N (a - a_j)\| + \|P_N a_j - a_j\| + \|a_j-a\|\\
&\leq (1+c) \|  a -  a_j \| + \max_{i=1, \ldots, m}\|P_N a_i - a_i\|\\
&\leq (1+c) \delta + \max_{i=1, \ldots, m}\|P_N a_i - a_i\|,
\end{align*}
where $c$ is the bound of the operator norms related to $P_N$.
Then $\sup_{a\in K} \|P_N a -a \| \leq (1+c) \delta + \max_{i=1, \ldots, m}\|P_N a_i - a_i\|$ and so taking the lim sup on both sides we get 
\[
\limsup_{N\to\infty}\sup_{a\in K} \|P_N a -a \| \leq (1+c) \delta  
\]
since $ \lim_{N\to\infty}\|P_N a_i - a_i\|=0$ for all $i\in\{1, \ldots, m\}$. By the fact that $\delta $ is arbitrary we get 
\[
\lim_{N\to\infty} \sup_{a\in K} \|P_N a -a\|=0
\]
as wanted.
\end{proof}

Before proving Lemma \ref{lm: density} we introduce  the  Haar wavelet functions and illustrate their use within the context of fractional Sobolev spaces $H^s_r$.
For simplicity of notation we recall only the case of the Haar wavelets on $ \mathbb R$ (see \cite{triebel10}, Section 2.2, eqn (2.93)--(2.96)) and leave to the reader the extension to $ {\mathbb R}^d$ which can be found in Section 2.3 of the same book.  
We define
\[
h_M(x) := 
\begin{cases}
1  & \text{ if } 0\leq x <\frac12,\\
-1 & \text{ if }  \frac12\leq x<1,\\
0  & \text{ if }  x\notin[0,1),
\end{cases}
\]
\[
h_F(x)  := |h_M(x)|, \quad h_{-1,m}(x) := \sqrt 2 h_F(x-m), \; m\in \mathbb Z,
 \]
and 
\[ 
h_{j,m}(x) := h_M(2^j x -m), \; j\in N_0, m\in \mathbb Z.
\]
Then the family
\begin{equation}\label{eq: Haar basis}
\left\{h_{j,m}, j\in \mathbb N_0 \cup \{-1\}, m\in \mathbb Z\right\}
\end{equation}
is an unconditional basis of $H^s_r(\mathbb R)$ for $2\leq r\leq \infty$ and $-\frac12<s<\frac1r$ by \cite[Theorem 2.9, (ii)]{triebel10}. Note that $r=\infty$ is included here  but is not included in Lemma \ref{lm: density} because of Step 1 in the proof below.  The analogous result in dimension $d\geq1$ is given in Theorem 2.21, (ii). Moreover for any $h\in H^s_r(\mathbb R) $ we have the unique representation
\[
h = \sum_{j=-1}^\infty \sum_{m\in \mathbb Z} \mu_{j,m} 2^{-j (s-\frac1r)} h_{j,m} 
\]
where 
\begin{equation}\label{eq: coefficients mu}
 \mu_{j,m}:= 2^{j (s-\frac1r+1)} \int_{\mathbb R} h(x) h_{j,m}(x) \mathrm d x,
\end{equation} 
 and the integral has to be interpreted as a dual pairing as mentioned in \cite[Theorem 2.9]{triebel10}, see also \cite[Remark 1.14]{triebel08}. 
Rewriting the same series  with a different notation 
 $\bar \mu_{j,m}:= 2^{j } \int_{\mathbb R} h(x) h_{j,m}(x)\mathrm dx $ 
we get another equivalent representation for $h$ given by
\begin{equation}\label{eq: Haar repres}
h = \sum_{j=-1}^\infty \sum_{m\in \mathbb Z} \bar \mu_{j,m}  h_{j,m}. 
\end{equation}
Defining the projector $P_N$ as
\begin{equation}\label{eq: Haar repres cut at N}
P_N h:= \sum_{j=-1}^N \sum_{m= -N}^{N} \bar \mu_{j,m}  h_{j,m}, 
\end{equation}
for $h$ of the form \eqref{eq: Haar repres}, then clearly $P_Nh \in H^{s}_{r}(\mathbb R)$ and 
\begin{equation}\label{eq: Haar convergence}
\|h-P_N h \|_{H^{s}_{r}(\mathbb R^d)} \to 0,
\end{equation}
  as $N\to \infty$.

\begin{rem}\label{rm: family PN} 
We observe that the projector $P_N$ enjoys the bound
\[
\|P_N h\|_{H^{s}_r(\mathbb R)} \leq \|h\|_{H^{s}_r(\mathbb R)}. 
\]
This can be seen as follows. We denote by $\mu(h)$ the collection of $\mu_{j,m}$ given by \eqref{eq: coefficients mu} for some $h$. Then for $2\leq r\leq \infty $ the map $h \mapsto \mu(h)$ is an isomorphism between $H^s_r$ and   $f^-_{r2}$, where the latter is a  space of sequences.   For a precise definition of $f^-_{r2}$, its norm and the statement of this isomorphism property,  see \cite{triebel10} in particular,  see Section 2.2.3, Theorem 2.9 for the 1-dimensional case and Section 2.3.2, Theorem 2.21 for the $d$-dimensional one. Moreover  the sequence of coefficients $\mu(P_N h) $ coincide with $\mu(h)$ for all $j, |m|>N$ and is zero otherwise. Thus by definition
of the norm of $f^-_{r2}$, we have  $\|\mu_{j,m}(P_N h)\|_{f^-_{r2}}\leq \|\mu_{j,m}(h)\|_{f^-_{r2}}$ and this together with the isomorphism implies  $\|P_N h\|_{H^{s}_r(\mathbb R)} \leq \|h\|_{H^{s}_r(\mathbb R)} $ as stated.
  \end{rem}

\begin{proof}[Proof of Lemma \ref{lm: density}]
We will show that the dense inclusion holds for real-valued functions, namely that $ C_c^\infty([0,T]\times \mathbb R^d)\subset C([0,T]; H^{s}_{r}(\mathbb R^d)) $. To get the full statement  it is then enough to apply this result to each component of functions in $ C([0,T]; H^{s}_{r} ) $ .\\
\emph{\textbf{Step 1.} Density of $C_c^\infty(\mathbb R^d)$ in $H^{s}_{r}(\mathbb R^d)$.} 
It is a known result that $C_c^\infty(\mathbb R^d)$ is dense in $H^{s}_{r}(\mathbb R^d)$ for all $1<r<\infty $ and $-\infty<s<\infty$. For a proof  see for example \cite[Theorem in Section 2.3.2, part (b)]{triebel78}.
\\
\emph{\textbf{Step 2.} Non-smooth approximating sequence for $l\in C([0,T]; H^{s}_{r}(\mathbb R^d) )$.}
We consider $d=1$ in the proof for simplicity of notation and explanation, but the same methodology extends to the case $d\geq1$, see for example \cite[Section 2.3.1]{triebel10}. We will use here the notation of Section 2.2.2  in the same book, which deals with the case $d=1$, in particular let $\{h_{j,m}, j\in \mathbb N_0\cup \{-1\}, m\in \mathbb Z \}$ be the Haar basis on $L^2(\mathbb R)$ defined in \eqref{eq: Haar basis}. 
Now  let $l\in C([0,T]; H^{s}_{r}(\mathbb R) )$ and let $t\in[0,T]$. 
 We recall that  $(P_N)_N$ defined by \eqref{eq: Haar repres cut at N} is a family of linear  operators  acting on $H^{s}_{r}(\mathbb R)$. 
The coefficients $\bar \mu$ of $P_Nl(t)$ are now parametrized by time, namely  $\bar \mu_{j,m}(t)= 2^{j } \int_{\mathbb R} l(t,x)  h_{j,m}(x)\mathrm dx$.  By  \eqref{eq: Haar convergence} we have that $P_N l(t)\to l(t)$ in $H^{s}_r(\mathbb R)$ as $N\to \infty$, for all $t\in [0,T]$. 
It is clear by definition of the coefficients that $t\mapsto \bar \mu_{j,m}(t)$ is continuous and each term $ t \mapsto \bar \mu_{j,m} (t) h_{j,m}$  in the finite sum belongs to
$C([0,T]; H^{s}_{r}(\mathbb R) )$
hence  $P_N l\in C([0,T]; H^{s}_{r}(\mathbb R) )$. 
We will now show 
 that $P_N l\to l$ in $  C([0,T]; H^s_r(\mathbb R) ) $, namely that 
  \begin{equation}\label{eq: lN to l}
\lim_{N\to \infty} \sup_{t\in[0,T]} \|l(t)-P_N l(t)\|_{H^{s}_{r}(\mathbb R)}  =0.
  \end{equation}
  To prove this, we want to use  Lemma \ref{lm: projection converge uniformly in compacts} with the compact $K:=\{l(t): t\in[0,T]\}\subset H^{s}_{r}(\mathbb R)$  and the projection $P_N$  defined by \eqref{eq: Haar repres cut at N}.
   The family of functions $t\mapsto P_N l(t) $ is bounded in $N$ in the space $C([0,T]; H^{s}_r(\mathbb R))$ by Remark \ref{rm: family PN}. Since $\{l(t), t\in[0,T]\}$ is a compact set in $H^s_r$ then we can apply Lemma \ref{lm: projection converge uniformly in compacts} and we get \eqref{eq: lN to l}. \\
\emph{\textbf{Step 3.} Smoothing of non-smooth approximating sequence.}
The last step consists in showing that for any $l_N (t) := P_Nl(t)$ from Step 2 and for any $\varepsilon>0$, we can find an element $t\mapsto \tilde l_N (t)$ which is an element of $   C_c^\infty([0,T]\times\mathbb R)$ and such that $\sup_{t\in[0,T]}\|l_N(t)-\tilde l_N(t)\|_{H^s_{r}(\mathbb R)} <\varepsilon$. Then this would conclude the argument and show the density of $C_c^\infty([0,T]\times\mathbb R)$ in $C([0,T]; H^{s}_{r}(\mathbb R))$.

To find $\tilde l_N(\cdot)$  we observe that  $l_N(t)$ is a finite sum of terms of the type $\mu_{j,m}(t) h_{j,m}$, where the $\mu$s are continuous in time and $h_{j,m}$ is an element of the Haar basis. For each of this terms using Step 1 we can find $\tilde h_{j,m}\in C_c^\infty(\mathbb R)$ such that 
\[
\|h_{j,m}- \tilde h_{j,m}\|_{H^s_r({\mathbb R})} <\frac{\varepsilon}{\max_{t\in[0,T]} \sum_{j,m}|\mu_{j,m}(t)|},
\] 
where the sum appearing in the denominator is over the finite set of indices $j\in \{ -1, 0, \ldots, N\}$ and $m\in \{ -N,\ldots, 0, \ldots, N\}$. Then we set
\[
\tilde l_N(t): = \sum_{j,m} \mu_{j,m}(t) \tilde h_{j,m},
\]
where again the sum over $j,m$ is a finite sum.
Then for any $t\in[0,T]$ we have
\begin{align*}  
\|\tilde l_N(t) - l_N(t)\|_{H^{s}_{r}(\mathbb R)} & = \| \sum_{j,m} \mu_{j,m}(t)(h_{j,m}- \tilde h_{j,m} )\|_{H^{s}_r(\mathbb R)} \\
&  \leq \max_{t\in[0,T]}\sum_{j,m}|\mu_{j,m}(t)| \| (h_{j,m}- \tilde h_{j,m} )\|_{H^{s}_{r}(\mathbb R)} \\
&<\varepsilon. 
\end{align*}
\end{proof}
Below we given the proof of Theorem \ref{thm: A(b)=int b for Lp functions}.

\begin{proof}[Proof of Theorem \ref{thm: A(b)=int b for Lp functions}]
The proof is split in two steps. In Step 1 we show that \eqref{eq: A(b)=int b for Lp functions} holds for  $g\in C([0,T]; H^{s}_r)\cap L^\infty_{\text{loc}}([0,T]\times \mathbb R^d; \mathbb R^d)$ and bounded  
 functions with compact support. In Step 2 we treat the general case.
 
The proof is written for real-valued functions, and can be applied component by component.\\
{\bf Step 1.} \emph{$g$ bounded  
function with compact support.}\\
We consider a sequence $\phi_N:\mathbb R^d \to \mathbb R $ of mollifiers converging to the Dirac measure and for each $N$ we define an  operator $P_N$ acting on $h  \in  H^s_r(\mathbb R^d) $ by
$$
P_N h :=  (h * \phi_N).
$$
It is easy to show that for every $h\in H^s_r(\mathbb R^d)$ then  $P_N$ and $A^{-s/2}:= (I-\frac12\Delta)^{-s/2}$ commute, that is 
\begin{equation}\label{eq: commutation}
P_N(A^{-s/2} h) = A^{-s/2}(P_N h) .
 \end{equation}
Indeed by the definition of the norm in the $H^s_r(\mathbb R^d)$-spaces, we have $A^{-s/2}h\in L^r(\mathbb R^d)$. Denoting by  $\mathcal F$ the Fourier transform in $L^r(\mathbb R^d)$ we have
\begin{align*}
\mathcal F \left( A^{-s/2} (P_N h) \right)(\xi) 
&= \left(1+\frac{\xi^2}{2}\right)^{-s/2} \mathcal F (P_N h )(\xi)  \\
&=   \left(1+\frac{\xi^2}{2}\right)^{-s/2} \mathcal F (h )(\xi) \mathcal F (\phi_N)(\xi)  \\
&= \mathcal F\left(A^{-s/2} h \right) (\xi) \mathcal F (\phi_N)(\xi)   .
\end{align*}
 Taking the inverse Fourier transform on both sides we obtain the commutation property as stated in \eqref{eq: commutation}.  Now  it easily follows that 
\begin{equation}\label{eq: PN converges} 
 P_N h \to  h \text{ in } H^s_r(\mathbb R^d), \text{ as } N\to \infty
 \end{equation}
for every $ h \in H^s_r(\mathbb R^d)$,
  using the definition of the norm in the fractional Sobolev spaces, the property that $P_N f \to f$ in $L^r(\mathbb R^d)$ for $f$ in the latter space (in particular for $f= A^{-s/2}h$) and the commutation property \eqref{eq: commutation}. Moreover $P_N $ is a contraction in the same spaces, namely 
 \begin{equation}\label{eq: PN is a contraction}
 \|P_N h\|_{H^s_r(\mathbb R^d)} \leq  \| h\|_{H^s_r(\mathbb R^d)}.  
 \end{equation}
This can be seen by observing that $$
 \|P_N h\|_{H^s_r(\mathbb R^d)} =  \|A^{-s/2}( P_N h)\|_{H^0_r(\mathbb R^d)} = \| P_N (A^{-s/2}h)\|_{H^0_r(\mathbb R^d)},$$ 
where we have used \eqref{eq: commutation}, and the latter is bounded by $  \|A^{-s/2} h\|_{H^0_r(\mathbb R^d)} =   \| h\|_{H^s_r(\mathbb R^d)} $ because $P_N$ is a contraction operator in $H^0_r(\mathbb R^d)=L^r(\mathbb R^d)$.  
Property \eqref{eq: PN is a contraction} is applied to $h=g(t, \cdot)$ for all $t\in[0,T]$ to show that the function $t\mapsto P_Ng(t,\cdot)$ is continuous from $[0,T]$ to $H^s_r(\mathbb R^d)$. Indeed for any sequence $t_k\to t$ we have
\begin{align*}
\| P_Ng(t_k, \cdot)-P_Ng(t, \cdot)\|_{H^s_r(\mathbb R^d)}
& = \| P_N (g(t_k, \cdot)-g(t, \cdot))\|_{H^s_r(\mathbb R^d)}\\
& \leq \| g(t_k, \cdot)-g(t, \cdot)\|_{H^s_r(\mathbb R^d)},
\end{align*}
which goes to zero by  assumption on $g$. 
To show that 
\begin{equation}\label{eq: PN g to g}
P_Ng\to g
\end{equation}
 in $C([0,T]; H^s_r(\mathbb R^d))$, we use Lemma \ref{lm: projection converge uniformly in compacts} with 
$H= H^s_r(\mathbb R^d)$. 
 We can do so since the family of operators $(P_N)_N$ is linear and equibounded in $H^s_r(\mathbb R^d)$ by \eqref{eq: PN is a contraction}, and it fulfils
\eqref{eq: PN converges}.
 Thus defining the compact $K$ in $H^s_r(\mathbb R^d)$ by $K:= \{ g(t): t\in[0,T]\}$ we have 
\[
\sup_{a\in K} \|P_N a - a\|_{ H^s_r(\mathbb R^d)} = 
\sup_{0\leq t\leq T} \| P_Ng(t, \cdot)-g(t, \cdot)\|_{H^s_r(\mathbb R^d)} 
\]
and by Lemma \ref{lm: projection converge uniformly in compacts} the quantity above converges to 0 as $N\to\infty$.
At this point we observe that $P_N g \in C_c([0,T]\times \mathbb R^d; \mathbb R^d)$ because both $g$ and $\phi_N$ have compact support. Therefore $A^{W,W}(P_N g)$ is well-defined and \eqref{eq: A(b)=int b for Lp functions} holds for $g$ replaced by $P_N g $ thanks to \eqref{eq: operator AWW}. Moreover by \eqref{eq: PN g to g} we can apply Remark \ref{rm: continuity of A WW in H}, part 2.\ and get
\begin{equation}\label{eq: xx}
\lim_{N\to\infty} A^{W,W}(P_Ng) = A^{W,W}(g) \text{ in } \mathcal C.
\end{equation}
Finally we can see that  
\begin{equation}\label{eq: xxx}
\int_0^\cdot P_Ng(s, W_s)\mathrm ds \to  \int_0^\cdot  g(s, W_s)\mathrm ds 
\end{equation}
\emph{u.c.p} when $N\to \infty$. Indeed 
\begin{align} \label{E24}
E&\left[\sup_{0\leq t \leq T}|\int_0^t (P_Ng-g)(s, W_s)\mathrm ds|\right] \nonumber\\
&\leq E\left[\sup_{0\leq t \leq T}\int_0^t |P_Ng-g|(s, W_s)\mathrm ds\right]\\
&\leq \int_0^T\int_{\mathbb R^d} |P_Ng-g|(s, y)p_s(y)\mathrm dy\mathrm ds,\nonumber
\end{align}
where $p_s(y)$ is the mean-zero Gaussian density in $\mathbb R^d$ with variance $s$.
Now for almost all $(s,x)\in[0,T]\times \mathbb R^d$ we have $|P_Ng(s,x)|\leq \|g\|_{L^\infty([0,T]\times\mathbb R^d)}\leq C$, because  $g$ is bounded by assumption. This, together with the fact that 
\[
\int_{[0,T]\times \mathbb R^d} p_s(y) \mathrm ds \mathrm dy = T
\]
implies that \eqref{E24} is bounded by $2CT$.
Moreover for almost all $(s,x)\in[0,T]\times \mathbb R^d$ we also have 
\[
(P_Ng-g)(s,x) \to 0.
\]
By Lebesgue dominated convergence theorem the RHS of \eqref{E24} converges to 0. This implies \eqref{eq: xxx} and with \eqref{eq: xx} we conclude. \\
{\bf Step 2.} \emph{General case $g\in C([0,T]; H^{s}_r(\mathbb R^d))\cap L^\infty_{\text{loc}}([0,T]\times \mathbb R^d)$.}\\
Let us define $\tau_M:= \inf\{t\geq 0 \text{ such that } |W_t|>M\}$. Clearly $\tau_M\to \infty$ a.s.\ as $M\to \infty$. Moreover we define a family of smooth functions 
\[
\chi_M(x) =\begin{cases}
1 & \text{ if } |x|\leq M\\
0 & \text{ if } |x|\geq M+1
\end{cases}
\]
and with $0\leq\chi_M(x)\leq 1$. Then we set $g_M(s,x) := g(s,x)\chi_M(x)$. It is clear that $g_M(s, W_s) = g(s, W_s)$ for all $\omega$ and for all $s\leq t\wedge \tau_M$ for any arbitrary $t$, hence
\begin{equation}\label{eq: approx1}
\int_0^{t\wedge \tau_M} g(s, W_s) \mathrm ds = \int_0^{t\wedge \tau_M} g_M(s, W_s) \mathrm ds.
\end{equation}
On the other hand we know that $g_M$ is bounded and has compact support by definition, and that $g_M\in C([0,T]; H^{s}_r(\mathbb R^d))$ because $g$ is in the same space and $\chi_M$ is smooth (using the pointwise multipliers property, see \cite[Section 2.2.2]{triebel08}). 
So Step 1 applies to $g_M$ 
\begin{equation*}
A^{W,W} (g_M) = \int_0^{\cdot} g_M(s, W_s) \mathrm ds.
\end{equation*}
and in particular it holds for the time $t\wedge \tau_M$, that is 
\begin{equation}\label{eq: approx2}
A^{W,W}_{t\wedge \tau_M} (g_M) = \int_0^{t\wedge \tau_M} g_M(s, W_s) \mathrm ds.
\end{equation}
Now we want to show that 
\begin{equation}\label{eq: Aww aim}
A^{W,W}_{\cdot\wedge \tau_M} (g_M) = A^{W,W}_{\cdot\wedge \tau_M} (g).
\end{equation}
 To this aim, let us consider an approximating sequence $(g^n)_n$ of $g$ in $ C_c ([0,T]\times\mathbb R^d)$, which exists due to Lemma \ref{lm: density}. Then we set $g_M^n:=g^n\chi_M$ for each $n$, and this is
 an approximating sequence for $g_M$ in $C([0,T]; H^s_r(\mathbb R^d))$.  Indeed  the linear map
 $\phi\mapsto  \phi\chi_M$   is continuous in $C([0,T];H^{s}_r(\mathbb R^d))$ by \cite[equation (2.50)]{triebel08},   namely there exists a constant $c(M)$ only dependent on $\chi_M$ such that 
 \[
\|\phi\chi_M\|_{H^s_r(\mathbb R^d)} \leq c(M) \|\phi\|_{H^s_r(\mathbb R^d)}. 
\]
Then
\begin{align*}
\|g^n_M-g_M\|_{C([0,T]; H^s_r(\mathbb R^d))} 
& = \|g^n\chi_M-g\chi_M\|_{C([0,T]; H^s_r(\mathbb R^d))}\\
& = \sup_{0\leq t\leq T} \|(g^n(t, \cdot)-g(t, \cdot))\chi_M\|_{H^s_r(\mathbb R^d)}\\
& \leq  c(M) \sup_{0\leq t\leq T} \|g^n(t, \cdot)-g(t, \cdot)\|_{H^s_r(\mathbb R^d)}\\
&= c(M) \|g^n-g\|_{C([0,T]; H^s_r(\mathbb R^d))}, 
\end{align*}
and since $g^n$ converges to $ g$ in $C([0,T]; H^s_r(\mathbb R^d))$ then so does $g^n_M$ to $g_M$. \\
For each $n$  we have
\begin{equation}\label{eq:  Aww truncated}
 A^{W,W}_{\cdot\wedge \tau_M} (g^n_M) = A^{W,W}_{\cdot\wedge \tau_M} (g^n) 
\end{equation}
because both sides are defined explicitly and the two functions coincide  before $\tau_M$. We note that $ A^{W,W}_{\cdot } (g^n_M) $ (resp. $ A^{W,W}_{\cdot   } (g^n)$) converges u.c.p.\ to  $ A^{W,W}_{\cdot } (g_M) $ (resp. $ A^{W,W}_{\cdot   } (g)$)  as $n\to \infty$. The truncated processes, which are the  left-hand side and the right hand-side of \eqref{eq:  Aww truncated} also converge u.c.p., hence we get \eqref{eq: Aww aim}.
This, together with \eqref{eq: approx1} and \eqref{eq: approx2} gives 
\begin{equation}\label{eq: approx3}
  \int_0^{\cdot\wedge \tau_M} g(s, W_s) \mathrm ds = A^{W,W}_{\cdot\wedge \tau_M} (g). 
\end{equation}
For almost all  $\omega$ there exists $n_0(\omega)$ such that for all $M>n_0(\omega)$ we have $ \tau_M(\omega)\geq T$, then taking the limit as $M\to\infty$  of \eqref{eq: approx3} we conclude.
\end{proof}

\section*{Acknowledgments}

The authors are grateful to the Referee for the careful reading
and the stimulating comments.
The contribution of the first named author  was partially supported by the LMS grant 41501 ``research in pairs''.
The work of the second named author
was supported by a public grant as part of the
{\it Investissement d'avenir project, reference ANR-11-LABX-0056-LMH,
  LabEx LMH,}
in a joint call with Gaspard Monge Program for optimization, operations research and their interactions with data sciences.
That author  was also partially supported by the grant 346300 for IMPAN from the Simons Foundation and the matching 2015-2019 Polish MNiSW fund.

\bibliographystyle{plain}
\bibliography{biblio-BSDE}

\end{document}